\documentclass{amsart}
\usepackage{amsmath}%
\usepackage{amsfonts}%
\usepackage{amssymb}%
\usepackage{graphicx}
\usepackage{srcltx}%

\newtheorem{theorem}{Theorem}
\theoremstyle{plain}

\newtheorem{lemma}{Lemma}

\newtheorem{proposition}{Proposition}
\newtheorem{remark}{Remark}

\numberwithin{equation}{section}

\newcommand{\an}{\alpha_N}

\newcommand{\hn}{\mathcal{S}_N}

\newcommand{\qw}{\begin{equation}}
\newcommand{\qwe}{\end{equation}}

\newcommand{\rn}{r(N)}
\newcommand{\tn}{t(N)}
\newcommand{\ann}{\alpha_N^2}
\newcommand{\ort}{\mathbb{E}}
\newcommand{\comnd}{{{N}\choose{d}}}
\newcommand{\rsn}{\bar{S}_N}

\def \R {{\mathbb R}}

\begin{document}
\title[Universality and extremal aging for Dynamics of Spin Glasses]{Universality and extremal aging for Dynamics of Spin Glasses on sub-exponential time scales}
\author{G\'{e}rard Ben Arous, Onur G\"{u}n}
\address
{G\'{e}rard Ben Arous, Courant Institute\newline
\indent 251 Mercer Street\newline
\indent New York, NY 10012--1185, USA}%
\email{benarous@cims.nyu.edu}%
\address
{Onur G\"{u}n, CMI-LATP, Universit\'{e} de Provence\newline
\indent 39 rue Joliot-Curie\newline
\indent F-13453 Marseille cedex 13, FRANCE}%
\email{ogun@latp-mrs.univ.fr}%
\urladdr{http://www.onurgun.com}
\date{November 29, 2009}%

\begin{abstract}
We consider Random Hopping Time (RHT) dynamics of the Sherrington - Kirkpatrick (SK) model and $p$-spin models of spin glasses. 
For any of these models and for any inverse temperature $\beta>0$ we prove that, on time scales that are sub-exponential in the dimension, the properly scaled {\it clock process} (time-change process) of the dynamics converges to an extremal process. Moreover, on these time scales, the system exhibits aging like behavior which we called {\it extremal aging}. In other words, the dynamics of these models ages as the random energy model (REM) does. Hence,
by extension, this confirms Bouchaud's REM-like trap model as a universal aging
mechanism for a wide range of systems which, for the first time, includes
the SK model.

\end{abstract}
\maketitle

\section{Introduction and Main Results}

Aging is one of the distinguishing features of the long-time behavior of the dynamics of a large class of important disordered systems, which includes mean-field spin glasses. 
Roughly, a system ages if its decorrelation properties are time-dependent: the older the system is, the longer it takes to forget its state, or equivalently, the system is more and more frozen as it ages.

The theoretical modeling of aging had a breakthrough with the introduction of a simple model, the trap model, by Bouchaud and Dean in the early 90s \cite{bus}, \cite{buspar}. 
In this effective model, traps, representing low energy configurations, reproduce the slow dynamics seen experimentally while transitions between these trapping states are reduced to those of a large complete graph. 
These simplifications allow an elementary detailed analysis. An almost universal aging mechanism,  \cite{arcsine}, has since emerged, based on this simple model, which has been proved to be valid very broadly and 
in particular for Random Hopping Time (RHT) dynamics of mean-field spin glasses (for a general view of trap models, not restricted to the case of dynamics of spin glasses, see the lecture notes \cite{lecnotes}). 
This aging mechanism is as follows: in a given long  time scale (long but still transient, i.e. shorter than the time to reach equilibrium)  the system wanders around among deep traps of a given depth scale, 
the time spent in shallower traps being negligible. The time spent in those deep traps sampled by the path of the dynamics behaves as a sum of independent heavy-tailed random variables, even though, a priori, 
trapping times are neither independent nor heavy-tailed. This is usually stated as the fact that the natural clock of the system converges to a stable subordinator. 
The aging properties are then seen as natural consequences of this convergence, through the classical arcsine law. This picture, which is universal i.e. model-independent, is of course expected to break down for time scales 
long enough to reach equilibrium. In those time scales, since the equilibrium properties depend on the model, the behavior of the dynamics should also depend on the model.

The universality of stable subordinators has been proved to hold for the RHT dynamics of the Random Energy Model (REM) in \cite{bbg1},\cite{bbg2}, \cite{arcsine}, \cite{cerg}, and for $p$-spin models with $p\geq 3$ in \cite{bbc}, 
for a broad range of time scales, i.e times scales $t(N)= e^{cN}$ which are exponential in the size $N$ of the system but shorter than the equilibration time of the system (i.e. $c$ should be appropriately small).

However, this does not include the important case of the Sherringhton-Kirkpatrick (SK) spin glass (the case $p$=2). 
The dynamics of the SK model on exponential time scales seems to belong to a different universality class. 
On the other hand, the static results about equilibrium REM universality proved in \cite{bovkur}, \cite{resamp}, \cite{resampoz} suggest that the dynamics of the SK model should have REM-like behavior 
when observed on sub-exponential time scales $t(N)= e^{o(N)}$. This is one of the results we obtain here. In fact, we consider here the more general question of the RHT dynamics of mean-field spin glasses 
on sub-exponential time scales, and show that they are universal. The limiting picture cannot be linked to an $\alpha$-stable subordinator, since here the index $\alpha$ should be zero. 
In those time scales the process spends most of its time in one trap, the deepest trap it finds. The clock process is now related directly to what we call the ``maximal process'' which is basically the time spent 
in the deepest trap met by the system at a given time. Our statements will rely on the natural notion of extremal processes instead of subordinators. We are then led to introduce a new notion of ``extremal aging'' 
well suited to these time-scales. 

In the rest of this introduction we describe the models of spin glasses of whose dynamics we are studying, and then give our main result about extremal aging. 
We then proceed to give the core result, which is the convergence of the suitably normalized clock process and of the maximal process. We end this introduction by giving an outline of the proofs.

\subsection{The Models}
\vspace{0.1in}
Let us describe more precisely the class of models we are considering. Our state space is the $N$-dimensional hypercube, 
$\hn=\{-1,+1\}^N$. The Hamiltonian of the SK model and the $p$-spin models at $\sigma\in \hn$ is given by $-\sqrt{N}H_N(\sigma)$ where
\begin{equation}\label{pcor}
H_N(\sigma)=\frac{1}{N^{\frac{p}{2}}}\sum_{1\leq i_1,\dots,i_p\leq N}J_{i_1,\dots,i_p}\sigma_{i_1}\cdots\sigma_{i_p},\; p\in\mathbb{N},\;p\geq 2
\end{equation}
with $J_{i_1,\dots,i_p}$ i.i.d. standard normal random variables. Here $p=2$ is the SK model and $p\geq 3$ is the $p$-spin models. We will denote by $\mathcal{H}$ the $\sigma$-algebra generated by random variables $H_N(\sigma),\;\sigma\in \mathcal{S}_N$. Then the Gibbs 
measure at inverse temperature $\beta$ is given by
\begin{equation}
\mu_{\beta,N}(\sigma)=Z_{\beta,N}^{-1}\exp(\beta\sqrt{N}H_N(\sigma)),
\end{equation}
where $Z_{\beta,N}$ is the partition function.

We define RHT dynamics (trap model dynamics) as a nearest-neighbor continuous time Markov chain $\sigma_N(\cdot)$ on $\hn$ with transition rates
\begin{equation}\label{trans}
w_N(\sigma,\tau)=\left\{\begin{array}{ll}N^{-1}e^{-\beta\sqrt{N}H_N(\sigma)},& \text{if }\text{dist}(\sigma,\tau)=1,\\0,&\text{otherwise,}\end{array}\right.
\end{equation}
where $\text{dist}(\sigma,\tau)=\#\{i:\;\sigma_i\not=\tau_i\}$ is the graph distance on the hypercube. In other words, $\sigma_N(t)$ waits at a site $\sigma$ an exponential time with mean $\exp(\beta\sqrt{N}H_N(\sigma))$ then moves to one of the neighbors of $\sigma$ uniform at random.

We will consider these dynamics on time scales $\tn$ that are sub-exponential in dimension. We choose
\begin{equation}
\tn=\exp(\an N)
\end{equation}
with
\begin{equation}
\an=N^{-c},\; c\in(0,1/2).
\end{equation}
\subsection{Universality of Extremal Aging}
We want to investigate aging properties of the RHT dynamics on sub-exponential time scales $\tn$. We choose our two-time correlation function to characterize aging as in \cite{bbc}: for any $t,s>0$ and $\epsilon\in(0,1)$ let $A_N^{\epsilon}(t,s)$ be the event that the fraction of spins flipped between times $t$ and $s$ is less than $\epsilon/2$, that is

\begin{equation}
A_{N}^{\epsilon}(t,s)=\{\text{dist}(\sigma_N(t),\sigma_N(s))\leq N\epsilon/2\}.
\end{equation}

Our main result shows a universal aging phenomena in these models for sub-exponential time scales.

\begin{theorem}\label{theop22}(Extremal Aging for SK and $p$-spin models)

For the SK and the $p$-spin models, for any $c\in(0,1/4)$, for all $\theta>0$ and $\epsilon\in(0,1)$, let
\qw
t_1(N)=\tn,\;\; t_2(N)=\tn(1+\theta)^{1/\an},
\qwe
then
\begin{equation}
\mathbb{P}[A_N^{\epsilon}\left(t_1(N),t_2(N)\right)]\overset{N\to\infty}\longrightarrow \left(\frac{1}{1+\theta}\right)^{1/\beta^2}.
\end{equation}
Moreover, if $p\not=3$ the same result holds for any $c\in(0,1/2)$.
\end{theorem}
\begin{remark}
The weaker result for $p=3$ is due to technical reasons and we do not believe that the $p=3$ case has a different behavior than the other models.
\end{remark}
\begin{remark}
The above result is also true for the RHT dynamics of the REM on the same time scales $\tn$ (see \cite{bg09}). Hence, the aging properties of the REM is universal for SK and $p$-spin models on sub-exponential time scales.
\end{remark}
\begin{remark}
Note that the ratio of the two times $t_2(N)/t_1(N)=e^{\theta N^c}$ diverges with $N$ but since $c\in(0,1/2)$ the logarithmic ratio $\log t_2(N)/\log t_1(N)$ converges to 1 as $N\to\infty$. 
Hence, we can think of the decorrelation result of Theorem \ref{theop22} as ``just before aging". We have called this type of decorrelation behavior \textbf{extremal aging}. 
The reason for this choice of name will become clear later (see Theorem \ref{theop1}).
\end{remark}
\subsection{Extremal Processes as a universal limit for maximal and clock processes}
The proof of Theorem \ref{theop22} relies on the fact that the trap model dynamics can be constructed as a 
random time-change of a simple random walk (SRW) on $\hn$. Our main tool to understand the RHT dynamics of these 
models is to study this time change process which is called the clock process. 
More precisely, let $Y_N(k)\in\hn,k\in\mathbb{N}$ denote the simple random walk on $\hn$ started from a point $Y_N(0)$ and let $\mathcal{Y}$ denote the $\sigma$-algebra generated by it. For $\beta>0$ we define the {\it clock process} $S_N(k),k\in\mathbb{N}$ by
\qw
S_N(k)=\sum_{i=0}^{k-1}e_i \exp(\beta\sqrt{N}H_N(Y_N(i))),
\qwe
where $(e_i,\;i\in\mathbb{N})$ is a sequence of i.i.d. mean one exponential random variables. Then $\sigma_N(\cdot)$ can be written as
\qw
\sigma_N(t)=Y_N(S_N^{-1}(t)).
\qwe

Let $\mathcal{E}$ denote the $\sigma$-algebra generated by the random variables $(e_i,\;i\in\mathbb{N})$.
We will assume that all the random variables are defined on a common abstract probability space $(\Omega,\mathcal{F},\mathbb{P})$. 
Note that the $\sigma$-algebras $\mathcal{H},\mathcal{Y}$ and $\mathcal{E}$ are independent under $\mathbb{P}$.

We also introduce a process which keeps record of the mean waiting time corresponding to the lowest energy found on the trajectory. For $\beta>0$ we define the {\it maximal process} $m_N(k),k\in\mathbb{N}$ by
\begin{align*}
m_N(k):&=\exp\left\{-\beta\min_{0\leq i\leq k-1}-\sqrt{N}H_N(Y_N(i))\right\}\\&=\exp\left\{\beta\sqrt{N}\max_{0\leq i\leq k-1} H_N(Y_N(i))\right\}\\&=\max_{0\leq i\leq k-1}\exp\left\{\beta\sqrt{N}H_N(Y_N(i))\right\}.
\end{align*}
We also set $m_N(0)=0$.

We are interested in the asymptotic properties of the clock process and the maximal process on time scales $\tn$. To this end we need to introduce another scale $\rn$ given by
\qw\label{anamanam}
\rn=\an^{-1}\beta^{-1}\sqrt{2\pi N}\exp(\ann\beta^{-2} N/2).
\qwe
$\rn$ will be seen as the proper scaling for the number of jumps of the process $\sigma_N$ in the time scale $\tn$. Since we are assuming $c\in(0,1/2)$ the above scale is sub-exponential. Note that, the exponential term $\ann\beta^{-2}N/2$ diverges only if $c<1/2$. That is the reason we have $1/2$ as a natural upper bound for $c$, otherwise the number of jumps scale is growing at most polynomially.  

\newcommand{\pro}{\mathbb{P}}

The following theorem is our main result about the convergence of the maximal and clock processes:

\begin{theorem}\label{theop1}(Convergence of the maximal and clock processes for SK and $p$-spin models)

For the SK model and the $p$-spin models, for any $c\in(0,1/4)$, under the conditional distribution $\pro(\cdot|\mathcal{Y})$, $\mathcal{Y}$ a.s.

{\it (i)}
\qw\label{ana}
\left(\frac{m_N(\cdot \rn)}{\tn}\right)^{\an}\overset{N\to\infty}\longrightarrow Y_{\beta}(K\cdot),
\qwe

{\it (ii)}
\qw\label{anaiki}
\left(\frac{S_N(\cdot \rn)}{\tn}\right)^{\an}\overset{N\to\infty}\longrightarrow Y_{\beta}(K\cdot)
\qwe
weakly on the space of c\`{a}dl\`{a}g functions on $[0,T]$ equipped with the $M_1$-topology where $Y_{\beta}(\cdot)$ is the
 extremal process generated by $G(x)=\exp(-1/x^{1/\beta^2}),\;x>0$ and
\qw\label{anacik}
K=2\beta^{-2} p.
\qwe
Moreover, if $p\not=3$ the same results hold for any $c\in(0,1/2)$.
\end{theorem}

\begin{remark}
The above Theorem is also true for the RHT dynamics of the REM for the time scale $\tn$. Hence, the REM dynamics picture is essentially universal for these models.
\end{remark}
\begin{remark}
For the RHT dynamics of REM the above theorem holds true with a slight difference in the number of jumps scale $\rn$. Specifically, in the REM dynamics, the corresponding number of jumps is $\an^2$ times $\rn$ of Theorem \ref{theop1}, [BG09]. This means that in order to find traps that are order of $\tn$ the SRW has to make more steps in the correlated  case than it needs to make in the independent case. Note that this was only a factor of of a constant for exponential time scales (see Theorem 1 in \cite{bbc} and Theorem 3.1 in \cite{arcsine}).
\end{remark}

We will explain in detail the $M_1$ topology in Section \ref{sectionfive}. 
Roughly, $M_1$-topology allows several big jumps made in a short time to produce one bigger jump, 
and as a result it is weaker than the usual Skorohord $J_1$-topology. 
Theorem \ref{theop1} is not true for $J_1$ topology. Due to the correlations in the energy landscape, 
neighbors of a deep point tend to be deep as well so that the clock process makes several consecutive large jumps. 
However, in the cases we study it turns out that these consecutive jumps are made in a very short time interval. 
Convergence in $J_1$ topology is sensitive to this kind of jumps made in very short time whereas convergence in $M_1$ topology is not. 
Naturally, for the REM model, where no correlations exist, one can expect convergence in $J_1$ topology and in fact we prove 
it in \cite{bg09}.

We will recall the definition of extremal processes in Section \ref{sectionfive}. 
One can think of an extremal process as a continuous version of a record process. 
It is natural that the maximal process $m_N$ converges to an extremal process. Theorem \ref{theop1} tells that the clock process is reduced to the contribution of the lowest energy found on the trajectory and converges to an extremal process as well.

\subsection{Discussion of the results}

Let us briefly discuss the results of Theorems \ref{theop22} and \ref{theop1}. In the language of trap models, 
a low energy state corresponds to a site with a deep trap. In the REM dynamics, on exponential time scales, 
the energy landscape explored by the dynamics is very heterogenous. The clock process is carried by the contributions 
from the deep traps found on the trajectory and it converges to an $\alpha$-stable subordinator, \cite{arcsine}. 
The same is basically true for the $p$-spin models on exponential time scales, the difference being that a deep trap consists of a valley of sites with low energies instead of a single site. However, the REM picture for the dynamics is not valid for the SK model (p=2) on these time scales.

In the REM dynamics, on sub-exponential time scales, eventually the deepest of these deep traps found on the trajectory dominates the clock process. Roughly speaking, in this case there are few deep traps and their depths are of the form $\tn x^{1/\an}$. As a consequence, the clock process has no non-trivial limit under any linear normalization. However, one can get a non-trivial limit by a non-linear normalization as in Theorem \ref{theop1}. Another consequence is that, after rescaling by $\tn$, the deepest trap dominates the clock process. This explains why we have same kind of convergence for the maximal and the clock processes. Briefly, it is enough to check the convergence of the maximal process in order to prove the convergence of the clock process. See \cite{g09} and \cite{bg09} for details. This picture is similar to the behavior of sums of i.i.d. random variables with slowly varying probability tails, see \cite{dar} and \cite{kas}.

Theorems \ref{theop22} and \ref{theop1} tell that the REM behavior on sub-exponential time scales is essentially valid for SK and $p$-spin models. 
Again, the difference is that a deep trap consists of a valley of sites with low energies instead of a single site. Moreover, we will see that the radius of these valleys are proportional to $\an^{-2}$.

\subsection{The Outline of the proofs}

The proof of Theorem \ref{theop1} basically follows the strategy of \cite{bbc}. Let us define
\qw
X_N^0(i):=H_N(Y_N(i)),\;i\in\mathbb{N}.
\qwe
Note that then $X_N^0$ is a Gaussian process parameterized by $\mathbb{N}$. It is easy to see from equation (\ref{pcor}) that
\qw
\mathbb{E}[X_N^0(i)X_N^0(j)]=\left(1-\frac{2\text{dist}(Y_N(i),Y_N(j))}{N}\right)^p.
\qwe
As explained above, the key part of Theorem \ref{theop1} is the convergence of the maximal process. 
Hence, we need to calculate statistics of the maximum of $X_N^0$. To do this, we pick another Gaussian process $X_N^1$ that 
has a simpler covariance structure that enables us to precise calculations about its extremes. 
Then, we compare the extremal statistics of $X_N^0$ and $X_N^1$ using Gaussian comparison techniques.

However, at the comparison stage we have an added difficulty. 
As mentioned earlier the number of jumps scale $\rn$ is larger in Theorem \ref{theop1} than in the REM case. 
The comparison arguments do not work with this scaling as we are comparing two Gaussian processes on a larger set. We come over this difficulty by a new re-sampling strategy.

We choose the auxiliary Gaussian process $X_N^1$ based on the following observations. In the time scales we are considering the trajectory of the SRW is locally very close to a straight line in the sense that: 
i) for times $t\leq \nu \sim N^{w},w<1$ the distance from the starting point grows essentially linearly with speed 1; ii) with a high probability the SRW walk will never return to a neighborhood of size $\nu$ of the starting point 
in $\rn$ number of steps. Next, we expect the energy landscape sampled by the SRW mainly consist of deep valleys whose statistics are asymptotically independent. Also, we expect that the SRW will be gone through a deep valley 
in $\nu$ number of steps for $\nu$ large enough. On the other hand, for sites inside a valley, by i) with a high probability $\text{dist}(Y_N(i),Y_N(j))=|i-j|$ and the covariance function $\mathbb{E}[X_N^0(i)X_N^0(j)]$ can be well 
approximated by the linear function $1-2p|i-j|/N$. Hence, we choose the replaced process $X_N^1$ as a block independent process with block size $\nu$ and with the linear covariance function $\mathbb{E}[X_N^1(i)X_N^1(j)]=1-2|i-j|/N$ for $i,j$ in the same block. 
This linear covariance structures allows us to calculate the extremal statistics in detail.

In order to prove Theorem \ref{theop22} we need to know more about how the jumps of the clock process occur. We will prove that if we coarse grain the clock process over blocks of size $o(N)$ the convergence statement of Theorem \ref{theop1} holds in $J_1$-topology. This means that jumps that are made in $\leq o(N)$ steps constitute a jump of the limiting process. Hence, during the time of one big jump only a negligible fractions of spins are flipped. We will actually prove a stronger version of Theorem \ref{theop22}:
\begin{theorem}\label{theo3}
Assume the hypothesis of Theorem \ref{theop1}. Under the conditional distribution $\pro(\cdot|\mathcal{Y})$, $\mathcal{Y}$ a.s.
\begin{equation}
\mathbb{P}[A_N^{\epsilon}\left(t_1(N),t_2(N)\right)|\mathcal{Y}]\overset{N\to\infty}\longrightarrow \left(\frac{1}{1+\theta}\right)^{1/\beta^2}.
\end{equation}
\end{theorem}
\begin{remark}
 Taking the expectation over $\mathcal{Y}$, Theorem 3 implies Theorem 1. 
\end{remark}
The rest of this paper organized as follows: in Section 2 we obtain the results needed for the auxiliary Gaussian process, in Section 3 we compare the real and the auxiliary Gaussian processes, Section 4 contains the random walk results and in Section 5 we present the proofs of the main theorems.

\section{Extremal statistics of the auxiliary Gaussian process}

\newcommand{\vv}{\nu}
In this section we investigate the extremal distributions of the block independent Gaussian process $X_N^1(i)$, $i\in\mathbb{N}$ where
\qw
\mathbb{E}[X_N^1(i)X_N^1(j)]=\left\{\begin{array}{ll}1-\frac{2p|i-j|}{N}&\lfloor i/\nu\rfloor =\lfloor j/\nu\rfloor,\\0 &\text{otherwise.}\end{array}\right.
\qwe
The block size $\nu$ is given by
\begin{equation}
 \nu=\lfloor N^\omega \rfloor,\;\;\;w\in(1/2+c,1)
\end{equation}
Recall that $\an=N^{-c}$, $c\in(0,1/2)$. Hence, $\nu$ satisfies
\begin{equation}\label{bukadar}
 N^{1/2}\an^{-1}\ll \nu\ll N,\;\;\;1\ll \nu \an^{2}.
\end{equation}

Using the block independence it is enough to study the extremal statistics inside a block. 
To this end we define the Gaussian process $U=\{U_i,i=1,...,\nu\}$ as a centered Gaussian 
process with covariance $\mathbb{E}[U_i U_j ]=1-2p|i-j|/N$. Then $X_N^1$ is $\rn/\nu$ independent copies of $U$.

\newcommand{\antk}{\an^2}
As mentioned above, we are interested in the statistics of the maximum of $\exp(\beta\sqrt{N}U_i)$ on the scale $\tn=\exp(\an N)$, under the non-linear normalization of taking the $\an$th power. We can see that
\qw
\exp(\beta\sqrt{N}U_i)\geq x^{1/\an}\tn \Longleftrightarrow U_i \geq \frac{\an}{\beta}\sqrt{N}+\frac{\log x}{\an\beta\sqrt{N}}.
\qwe
We define
\qw
C_N(x):=\frac{\an}{\beta}\sqrt{N}+\frac{\log x}{\an\beta\sqrt{N}}.
\qwe
The following proposition describes the statistics of the maximum of $U_i$ for the relevant level $C_N(x)$.

\begin{proposition}\label{esasli} For all $p\in\mathbb{N}$, uniformly for $x$ in compact subsets of $(0,\infty)$
\begin{equation}
 \lim_{N \to \infty}\frac{\rn}{\vv}\mathbb{P}(\max_{i=1,..,\vv}  U_i \geq C_N(x))=K/x^{1/\beta^2},
\end{equation}
where
\qw
  K=2\beta^{-2} p.
\qwe
\end{proposition}

As mentioned earlier we will compare the real and auxiliary Gaussian processes on a re-sampled set of indices. Now we describe the details of this re-sampling process inside a block. Let $(q_{i},\;i\in\mathbb{N})$ be a sequence of i.i.d. random variables with uniform distribution on $[0,1]$, independent from $U_i$'s. Let us denote by $\mathcal{U}$ and $\mathcal{W}$ the $\sigma$-algebras of $U_i$ and $q_i$, respectively. We assume that $\mathcal{U}$ and $\mathcal{W}$ is defined on the common probability space $\mathbb{P}$. Using $(q_{i},\;i\in\mathbb{N})$, we define the sequence of random variables $(w_{N,\rho}(i),i\in\mathbb{N})$ as $w_{N,\rho}(i)=1$ if $q_i\leq \rho\antk$ and $w_{N,\rho}(i)=0$ if $q_i>\rho\antk$. Thus, $(w_{N,\rho}(i),i\in\mathbb{N})$ is an i.i.d. sequence of Bernoulli random variables with
\begin{equation}
 \mathbb{P}(w_{N,\rho}(i)=1)=1-\mathbb{P}(w_{N,\rho}(i)=0)=\rho\antk.
\end{equation}
We want to investigate the maximum of $U_i$'s on the random set of indices defined by
\begin{equation}
w_{\rho}:=\{i\leq \nu:\;w_{N,\rho}(i)=1\}.
\end{equation}
In order to do this we need to know more about the number of $U_i$'s that are above the level $C_N(x)$.
\begin{proposition}\label{pasa} For all $p\in\mathbb{N}$ and $\rho>0$, there exists constants $C_1(\rho)=C_1(\rho;w,\beta,c,p)$ and $C_2(\rho)=C_2(\rho;w,\beta,c,p)$, such that uniformly for $x$ in compact subsets of $(0,\infty)$, for $N$ large enough
\begin{equation}
C_1(\rho)\frac{K}{x^{1/\beta^2}}\leq\frac{\rn}{\nu}\mathbb{E}\left[1-\exp\left\{-\rho\antk\sum_{i=1}^{\nu}\mathbf{1}\{U_i\geq C_N(x)\}\right\}\right]\leq C_2(\rho)\frac{K}{x^{1/\beta^2}},
\end{equation}
where $K=2\beta^{-2} p$ as in Proposition \ref{esasli}. Moreover, 
\begin{equation}
 \lim_{\rho\to\infty}C_i(\rho)=1,\;\;i=1,2.
\end{equation}
\end{proposition}
This proposition tells us that when the maximum of $U_i$'s is above than $C_N(x)$, roughly $\an^{-2}$ of $U_i$'s are also above $C_N(x)$. 
This explains why in the correlated models, in order to find traps of the order $\tn$, the SRW has to make $\an^{-2}$ times 
the number of steps needed in the independent case. That is the reason we choose $\antk$ as the density in the re-sampling scheme.
\newcommand{\kac}{\sum_{i=1}^{\nu}\mathbf{1}\{U_i\geq C_N(x)\}}

\begin{lemma}\label{onder} For all $p\in\mathbb{N}$, for any $\rho>0$, there exists a constant $C(\rho)$ s.t. uniformly for $x$ in compact subsets of $(0,\infty)$ for $N$ large enough
 \begin{equation}\label{hosaf}
  C(\rho)\frac{K}{x^{1/\beta^2}}\leq \frac{\rn}{\nu}\mathbb{P}(\max_{i\leq \nu,i\in w_{\rho}}  U_i \geq C_N(x))\leq \frac{K}{x^{1/\beta^2}},
 \end{equation}
where $K=2\beta^{-2} p$ is as in Proposition \ref{esasli}. Moreover we have,
\begin{equation}
 \lim_{\rho\to\infty}C(\rho)=1.
\end{equation}
\end{lemma}

The proof of Lemma \ref{onder} follows easily from this Propositions  \ref{esasli} and \ref{pasa}:

\begin{proof}[Proof of Lemma \ref{onder}]
It is clear that
\begin{equation*}
\mathbb{P}(\max_{i\leq \nu,i\in w_{\rho}}  U_i \geq C_N(x))\leq \mathbb{P}(\max_{i\leq \nu}  U_i \geq C_N(x)).
\end{equation*}
Then the upper bound follows from Proposition \ref{esasli}.

Using the Bernouilli distributions we have
\begin{equation}\label{cool}
 \mathbb{P}(\max_{i\leq \nu,i\in w_{\rho}} U_i\geq C_N(x)|\mathcal{U})=1-\left(1-\rho\antk\right)^{\sum_{i=1}^{\nu}\mathbf{1}\{U_i\geq C_N(x)\}}.
\end{equation}
Using the inequality $1-x\leq e^{-x},\;x\geq 0$ we have
\begin{equation*}
1-\left(1-\rho\antk\right)^{\sum_{i=1}^{\nu}\mathbf{1}\{U_i\geq C_N(x)\}}\geq 1-e^{-\rho\antk\kac}.
\end{equation*}
Hence, by (\ref{cool}) and Proposition \ref{pasa}, for $N$ large enough
\begin{align*}
 \frac{\rn}{\nu}\mathbb{P}(\max_{i\leq \nu,i\in w_{\rho}} U_i\geq C_N(x))\geq \frac{\rn}{\nu} \mathbb{E}[1-e^{-\rho\antk\kac}]\geq C_1(\rho)\frac{K}{x^{1/\beta^2}}
\end{align*}
where $C_1$ is as in Proposition \ref{pasa}. Thus, setting $C=C_1$ finishes the proof.
\end{proof}

The rest of this section is devoted to the proofs of Proposition \ref{esasli} and Proposition \ref{pasa}.

\newcommand{\rnreal}{\an^{-1}\beta^{-1}\sqrt{2\pi N}\exp(\antk\beta^{-2}N/2)}

\begin{proof}[Proof of Proposition \ref{esasli}]
Recalling that $\rn=\rnreal$ the statement of Proposition \ref{esasli} is equivalent to

\begin{equation}\label{ilker}
 \lim_{N \to \infty} \frac{\beta^{-1}\sqrt{2\pi N}e^{\an^2 \beta^{-2} N/2}}{\an \vv}\mathbb{P}\left(\max_{i=1,..,\vv} U_i \geq \frac{\an}{\beta} \sqrt{N} +\frac{\log{x}}{\an \beta \sqrt{N}}\right)=K/x^{1/\beta^2}.
\end{equation}
It is a well-known fact (see e.g. \cite{sle}) that random variables $U_i$ can be expressed using a sequence of i.i.d. standard normal random variables $Z_i$. $U_i$'s can be written as
\begin{equation}
 U_i=\Gamma_1 Z_1+\dots+\Gamma_i Z_i-\Gamma_{i+1}Z_{i+1}-\dots -\Gamma_{\vv}Z_{\vv},
\end{equation}
where
\begin{equation}
 \Gamma_1=\sqrt{1-\frac{p}{N}(\vv-1)} \;\; \text{ and  } \Gamma_2=\dots=\Gamma_{\vv}=\sqrt{\frac{p}{N}}.
\end{equation}
Observe that $\sum_{i=1}^{\vv}\Gamma_i^2 =1$. Let us define $G_i(z)=G_i(z_1,\dots,z_{\vv})$ as
\begin{equation}
 G_i(z)=\Gamma_1z_1+\dots\Gamma_i z_i-\Gamma_{i+1}z_{i+1}-\dots-\Gamma_{\vv}z_{\vv}.
\end{equation}
Hence the probability term in (\ref{ilker}) is equal to
\begin{equation}\label{anapar}
\int_{\mathbb{R}^\nu} \frac{\text{d}z}{(2\pi)^{\nu/2}}e^{-\frac{1}{2×}\sum_{i=1}^\nu z_i^2}\mathbf{1}\{\;\max_{i=1,\dots,\nu}G_i(z)\geq C_N(x)\}.
\end{equation}
\newcommand{\ger}{\mathbb{R}}
Note that since the distribution of Gaussian process is continuous, a.s. there exists only one maximum. We partition the domain of integration according to the index of the maximum of $G_i(z)$'s. Define
\begin{equation}
 D_k:=\{z\in\mathbb{R}^\nu:\; G_k(z)>G_i(z)\;\forall i\not= k\}.
\end{equation}
Then the integral (\ref{anapar}) is equal to
\begin{equation}\label{anapariki}
 \sum_{k=1}^\nu \int_{D_k}\frac{\text{d}z}{(2\pi)^{\nu/2}}e^{-\frac{1}{2×}\sum_{i=1}^\nu z_i^2} \mathbf{1}\{G_k(z)\geq C_N(x)\}.
\end{equation}
\newcommand{\alb}{\frac{\an}{\beta}}
On the set $D_k$ we do the following change of variables 
\begin{equation}
 \begin{array}{cc} z_i=b_i+\Gamma_i \alb\sqrt{N}&\text{if }i\leq k,\\ z_i=b_i-\Gamma_i \alb\sqrt{N}&\text{if }i>k.\end{array}
\end{equation}
Then
\begin{equation}\label{nealaka}
 G_i(z)=G_i(b)+(1-2|i-k|\frac{p}{N})\alb\sqrt{N}.
\end{equation}
It will be useful to define $\sum_{j=i+1}^k a_j$ as $\sum_{j=1}^k a_j -\sum_{j=1}^{i}a_j$ which is also meaningful for $i+1>k$. Using this definition
\begin{equation}\label{alakasiz}
 G_k(b)-G_i(b)=2\sum_{j=i+1}^k \Gamma_j b_j.
\end{equation}
Combining (\ref{nealaka}) and (\ref{alakasiz}) we have
\begin{align*}
 G_k(z)-G_i(z)&=2\sqrt{\frac{p}{N}}\sum_{i+1}^k b_j +2|i-k|\frac{\an p}{\beta\sqrt{N}},
\end{align*}
as a result $D_k$ is mapped to
\begin{equation}
 D_k'=\{b\in \ger^\nu:\;\sum_{j=i+1}^k b_j> -|i-k|\alb\sqrt{p}\;\forall i\not= k\}.
\end{equation}
Also,
\begin{equation*}
 -\frac{1}{2}\sum_{i=1}^\nu z_i^2 = -\frac{1}{2}\sum_{i=1}^\nu b_i^2 -G_k(b)\alb\sqrt{N} -\frac{\antk \beta^{-2} N}{ 2}.
\end{equation*}
Thus, after the change of variables (\ref{anapariki}) is equal to
\begin{equation}
 e^{-\antk \beta^{-2} N/2}\sum_{k=1}^\nu \int_{D_k'}\frac{\text{d}b}{(2\pi)^{\nu/2}}e^{-\frac{1}{2}\sum_{i=1}^\nu b_i^2}e^{-G_k(b)\alb \sqrt{N}}\mathbf{1}\{G_k(b)>\frac{\log x}{\an\beta\sqrt{N}}\}.
\end{equation}
Hence to finish the proof of Proposition \ref{esasli} we need to show that
\begin{equation}\label{hadihadi}
 \frac{\beta^{-1}\sqrt{2\pi N}x^{1/\beta^2}}{\an\nu}\sum_{k=1}^\nu \int_{D_k'}\frac{\text{d}b}{(2\pi)^{\nu/2}}e^{-\frac{1}{2}\sum_{i=1}^\nu b_i^2}e^{-G_k(b)\alb \sqrt{N}}\mathbf{1}\{G_k(b)>\frac{\log x}{\an\beta\sqrt{N}}\}
\end{equation}
is asymptotically equal to $K=2\beta^{-2}p$ for $x$ in a compact subset of $(0,\infty)$.

On $D_k'$ we do the change of variables $a_1=\an\beta\sqrt{N}G_k(b)$ and $a_i=b_i$ for $i\geq 2$, and hence,
\begin{equation}
b_1=\frac{a_1-\an\beta\sqrt{p}(a_2+\cdots-a_\nu)}{\an\beta\sqrt{N}\Gamma_1}.
\end{equation}
Denote by $D_k''$ the image of $D_k'$ under this change variables. We get a factor $1/{\an\beta\sqrt{N}}\Gamma_1$ from the Jacobian and we have
\begin{align}
\label{anane}(\ref{hadihadi})&=\frac{\beta^{-2} x^{1/\beta^2}\sqrt{2\pi}}{\antk \nu\Gamma_1}\sum_{k=1}^{\nu}\int_{D_k''} \frac{\text{d}a}{(2\pi)^{\nu/2}}e^{-\frac{1}{2}\sum_{i=2}^\nu a_i^2} e^{-b_1^2/2}e^{-a_1/\beta^2} \mathbf{1}\{a_1 > \log x\}\\\nonumber &=\frac{\beta^{-2} x^{1/\beta^2}\sqrt{2\pi}}{\antk \nu}\sum_{k=1}^{\nu}\int_{ D_k''} \frac{\text{d}a}{(2\pi)^{\nu/2}}e^{-\frac{1}{2}\sum_{i=2}^\nu a_i^2}\exp(-\frac{a_1}{\beta^2}-\frac{a_1^2}{2\Gamma_1^2\antk\beta^2 N}) \times\\\nonumber & \qquad \qquad \qquad \qquad\qquad\qquad\mathbf{1}\{a_1 > \log x\}\exp(-\frac{b_1^2}{2}+\frac{a_1^2}{2\Gamma_1^2\antk\beta^2 N})
\end{align}
The last exponential term
\begin{align*}
 -\frac{b_1^2}{2}+\frac{a_1^2}{2\Gamma_1^2\antk\beta^2 N}&=\frac{1}{2\Gamma_1^2\antk \beta^2 N }\times\\ &\left\{2\an\beta\sqrt{p}a_1(a_2+\cdots-a_\nu)-\antk\beta^2 p (a_2+\cdots-a_\nu)^2\right\}\underset{N\to\infty}\longrightarrow 0
\end{align*}
uniformly for all $|a_1|\leq \an N^{\frac{1+\delta}{2}}$ and $|a_2+\cdots-a_\nu|\leq \nu^{\frac{1+\delta}{2}}$ for $\delta>0$ small enough since $\nu\ll N$. The integration (\ref{anane}) 
over the rest of the domain can be bounded by $e^{-N^{\delta'}}$ for some $\delta'>0$ small enough, uniformly in $x$ for $x$ in a compact subset of $(0,\infty)$. Thus, up to an exponentially small error (\ref{anane}) is equal to
\begin{equation}\label{kalfa}
\frac{\beta^{-2} x^{1/\beta^2}}{\antk \vv\Gamma_1}\sum_{k=1}^{\nu}\int_{D_k''} \frac{\text{d}a}{(2\pi)^{(\nu-1)/2}}e^{-\frac{1}{2}\sum_{i=2}^\nu a_i^2} \exp(-\frac{a_1}{\beta^2}-\frac{a_1^2}{2\Gamma_1^2\antk\beta^2 N}) \mathbf{1}\{a_1 > \log x\}.
\end{equation}
Note that $D_k''$ does not depend on the first coordinate and can be written as $D_k''=\ger \times \bar{D}_k''$ where $\bar{D}_k''$ is the projection of $D_k''$ to the last $\nu-1$ coordinates. Hence, (\ref{kalfa}) is equal to
\begin{equation}
 \frac{\beta^{-2} x^{1/\beta^2}}{\antk \nu\Gamma_1} \left(\int_{\log x}^\infty {\text{d}a_1}\exp(-\frac{a_1}{\beta^2}-\frac{a_1^2}{2\Gamma_1^2\antk\beta^2 N})\right)\left(\sum_{k=1}^{\nu}\int_{\bar{D}_k''}\frac{\text{d}\bar{a}e^{-\frac{1}{2}\sum_{i=2}^\nu a_i^2}}{(2\pi)^{(\nu-1)/2}}\right)
\end{equation}
where $\bar{a}$ is the projection of $a$ onto the last $\vv-1$ coordinates.
Since $\antk N$ diverges with $N$, it is easy to see that the first integral converges to $\beta^2/ x^{1/\beta^2}$ as $N$ diverges uniformly in $x$ for $x$ in a compact subset of $(0,\infty)$. Also, observe that the second integral does not depend on $x$.
Finally, $\Gamma_1\to1$ as $N\to\infty$. Hence, to finish the proof of Proposition \ref{esasli} we need to show that 
\begin{equation}\label{emel}
 \frac{1}{\antk\vv}\sum_{k=1}^{\vv}\int_{\bar{D}''_k}\frac{\text{d}\bar{a}}{(2\pi)^{\frac{\vv-1}{2}}} e^{-\frac{1}{2}\sum_{i=2}^{\vv}{a_i^2}} \overset{N \to \infty} \longrightarrow K
\end{equation}
where $\bar{D}_k''=\{(a_2,\dots,a_{\nu})\in\ger^{\nu-1}:\;\sum_{j=i+1}^k a_j> -|i-k|\an\beta^{-1}\sqrt{p}\;\forall i\not= k \}$ and $K=2\beta^{-2} p$.
We use the fact that the integral in (\ref{emel}) can be related to random walk with drift. More precisely, define $V_N(0)=0$ and
\begin{equation}
V_N(k)=\sum_{i=1}^k (Z_i+\alb\sqrt{p}),\hspace{0.3in}k\in\mathbb{N},
\end{equation}
where $(Z_i,\; i\in\mathbb{N})$ is an i.i.d. sequence of standard normal random variables. In other words, $V_N$ is the random walk whose increments are i.i.d. normal random variables with mean $\an\beta^{-1}\sqrt{p}$ and variance 1.
For $k\geq 1$ define the events
\begin{equation}
 \{\tau_N=k\}:=\{V_1>0,\dots,V_{k-1}>0,V_k<0\}.
\end{equation}
$\{\tau_N=k\}$ is the event that the random walk $V_N$ goes below 0 first time in the $k$th step.
Using the definition of $\sum_{j=i+1}^k a_j$ and $\bar{D}_k''$ we have
\newcommand{\pron}{\mathbb{P}_N}
\begin{equation}
 \frac{1}{\antk\vv}\sum_{k=1}^{\vv}\int_{\bar{D}''_k}\frac{\text{d}\bar{a}}{(2\pi)^{\frac{\vv-1}{2}}} e^{-\frac{1}{2}\sum_{i=2}^{\vv}{a_i^2}}=\frac{1}{\antk \nu}\sum_{k=1}^{\nu}\pro(\tau_N\geq k)\pro(\tau_N \geq \nu-k).
\end{equation}
\newcommand{\taunt}{\hat{\tau}_N}
\newcommand{\taun}{\tau_N}

We need the following technical lemma.
\begin{lemma}\label{ilkfirs} Under the conditions of Proposition \ref{esasli}, there exist positive constants $K_1$ and $K_2$ s.t. as $N\to\infty$

\begin{align}
\label{iddabir}&{\bf (i)}\;\;\;\;\pro(\taun=\infty)\times \an^{-1} \longrightarrow K_1,
\text{ where } K_1=\beta^{-1}\sqrt{2p}.\\\label{ikifirs}
&{\bf (ii)}\;\;\; \mathbb{E}[\taun\;,\; \taun<\infty] \times \an \longrightarrow K_2.
\end{align}
\end{lemma}

\newcommand{\dnn}{d_N}
\newcommand{\patb}{-\sum_{k=1}^\infty \frac{s^k}{k}P(Z>\dnn \sqrt{k})}
\newcommand{\patbir}{\sum_{k=1}^\infty \frac{1}{k}P(Z>\dnn \sqrt{k})}
\newcommand{\pati}{\sum_{k=1}^\infty s^{k-1}P(Z>\dnn \sqrt{k})}
\newcommand{\patiki}{\sum_{k=1}^\infty P(Z>\dnn \sqrt{k})}
\newcommand{\pnk}{P(Z>\dnn \sqrt{k})}
\newcommand{\pnx}{P(Z>\dnn \sqrt{x})}
\newcommand{\pny}{P(Z>\dnn y)}

Let $\taunt(s)$ be the usual moment generating function of $\tau_N$, i.e.
\begin{align*}
 &\taunt(s)=\sum_{k=1}^{\infty} \pro(\taun=k) s^k.
\end{align*}
We define
\begin{equation}\label{kisa}
 d_N:=\alb\sqrt{p}.
\end{equation}
Due to a theorem by S. Andersen (Theorem 1, on page 413 of \cite{fel2}) we have
\newcommand{\pnz}{\pro(Z > \sqrt{p}\an\beta\sqrt{k})}
\begin{align}
 \nonumber\taunt(s)&=1-\exp\left(-\sum_{k=1}^{\infty}\frac{s^k}{k}\pro(V_N(k)<0)\right)\\ \nonumber &=1-\exp\left(-\sum_{k=1}^\infty\frac{s^k}{k}\pro(\sum_{i=1}^k Z_i < -k \alb\sqrt{p} )\right)\\ \label{asti}&=1-\exp\left(-\sum_{k=1}^\infty\frac{s^k}{k}P(Z > d_N\sqrt{k})\right)
\end{align}
where $Z$ is a standard normal random variable and $P$ is its probability distribution.
\begin{proof}[Proof of Lemma \ref{ilkfirs} part (i)]
Observe that,
\begin{equation}\label{hadibe}
 \pro(\taun=\infty )=1-\taunt(1)=\exp(-\sum_{k=1}^{\infty}\frac{1}{k}P(Z>d_N\sqrt{k}))
\end{equation}

Let us define random variables $Y_N:=Z^2/\dnn^{2}$. Note that  
\qw\label{korsan}
Y_N\overset{N\to\infty}\longrightarrow\infty\text{ a.s.}
\qwe
We have
\qw
\sum_{k=1}^{\infty}\frac{1}{k}P(Z>d_N\sqrt{k})=\frac{1}{2}\sum_{k=1}^{\infty}\frac{P(Z^2/\dnn^2\geq k)}{k}=\frac{1}{2}\sum_{k=1}^{\infty}\frac{P(Y_N\geq k)}{k}=\frac{1}{2}E\sum_{k=1}^{\infty}\frac{\mathbf{1}\{Y_N\geq k\}}{k}
\qwe
Let us define $\varphi(u)=\sum_{k=1}^{\lfloor u \rfloor}1/k$. Then
\qw
\sum_{k=1}^{\infty}\frac{1}{k}P(Z>d_N\sqrt{k})=\frac{1}{2}E\sum_{k=1}^{\infty}\frac{\mathbf{1}\{Y_N\geq k\}}{k}=\frac{1}{2}E(\phi(Y_N)\mathbf{1}\{Y_N\geq 1\})
\qwe
\newcommand{\wun}{\mathbf{1}\{Y_N\geq 1\}}
It is a well-known fact that
\qw
\varphi(u)-\log\lfloor u \rfloor \overset{u\to\infty}\longrightarrow \gamma
\qwe
where $\gamma$ is the Euler constant. Using (\ref{korsan}) we have a.s.
\qw
 \varphi(Y_N)-\log\lfloor Y_N \rfloor \overset{N\to\infty}\longrightarrow \gamma.
\qwe 
Using the bound
\qw
0\leq \varphi(u)-\log \lfloor u \rfloor\leq 1
\qwe
and (\ref{korsan}) we can conclude by the dominated convergence theorem that
\qw
E(\varphi(Y_N)\mathbf{1}\{Y_N\geq 1\})-E(\log\lfloor Y_N \rfloor \mathbf{1}\{Y_N\geq 1\})\overset{N\to\infty}\longrightarrow \gamma.
\qwe
It is easy to see that
\qw
E(\log\lfloor Y_N \rfloor \mathbf{1}\{Y_N\geq 1\})-E(\log Y_N \mathbf{1}\{Y_N\geq 1\})\overset{N\to\infty}\longrightarrow 0.
\qwe
It is clear by the definition of $Y_N$ that
\qw
E(\log Y_N \mathbf{1}\{Y_N\geq 1\})-2\log(\dnn^{-1})\overset{N\to\infty}\longrightarrow 2E(\log |Z|).
\qwe
Hence, we can conclude that
\qw
P(\taun=\infty)\times \dnn^{-1}\overset{N\to\infty}\longrightarrow \exp(-E(\log|Z|)-\gamma/2).
\qwe
and subsequently 
\qw
P(\taun=\infty)\times \an^{-1}\overset{N\to\infty}\longrightarrow \beta^{-1}\sqrt{p}\exp(-E(\log|Z|)-\gamma/2).
\qwe
This proves part (i) of Lemma \ref{ilkfirs} with $K_1=\beta^{-1}\sqrt{p}\exp(-E(\log|Z|)-\gamma/2)$. 

Now we calculate $K_1$. For $\alpha>0$ we define
\qw
V(\alpha)=\int_0^\infty e^{-x^2}x^{\alpha-1}dx.
\qwe
It is easy to see that
\qw
V'(\alpha)=\int_{0}^\infty e^{-x^2}(\log x )x^{\alpha-1}dx.
\qwe
We have
\qw\label{kediko}
V'(1)=\int_0^{\infty}e^{-x^2}\log x dx
\qwe
After the change variables $u=x^2$, $V(\alpha)$ is same as
\begin{equation}
 V(\alpha)=\frac{1}{2}\int_{0}^{\infty}e^{-u}u^{\alpha/2-1}du=\frac{1}{2}\Gamma(\alpha/2).
\end{equation}
Thus,
\begin{align}\label{paulhand}
 V'(1)=\frac{1}{4}\Gamma'(1/2)=\frac{1}{4}\Gamma(1/2)(\log \Gamma)'(1/2).
\end{align}
It is a well-known result that
\begin{equation}
 \psi(x)=\frac{\Gamma'(x)}{\Gamma(x)},
\end{equation}
where $\psi$ is the digamma function. Using the formula of $\psi$ for half-integer values we have
\begin{equation}
 \psi(1/2)=-\gamma-2\log 2,
\end{equation}
where $\gamma$ is the Euler constant. Using (\ref{paulhand}) and the fact that $\Gamma(1/2)=\sqrt{\pi}$ we get
\begin{equation}
 V'(1)=-\frac{\sqrt{\pi}}{4}(\gamma+2\log 2).
\end{equation}
Hence, by (\ref{kediko}) and the above equality, we can conclude by a change of variables that
\qw
E(\log |Z|)=-\frac{\log 2}{2}-\frac{\gamma}{2}.
\qwe
Hence, we have $K_1=\beta^{-1}\sqrt{2p}$
\end{proof}

\begin{proof}[Proof of part (ii) of Lemma \ref{ilkfirs}]
Using the moment generating function $\taunt$ we have
\begin{align*}
&E[\taun,\;\taun<\infty]=\frac{\partial \taunt(s)}{\partial s}|_{s=1}=\exp\left(-\sum_{k=1}^\infty\frac{1}{k}P(Z > d_N\sqrt{k})\right)\left\{\sum_{k=1}^{\infty}P(Z>\dnn \sqrt{k})\right\}
\end{align*}
We know by part (i) of the Lemma \ref{ilkfirs} that the exponential term above is asymptotically equivalent to $K_1\an$. Hence, to finish the proof it is enough to prove that
\qw
\frac{1}{\an^{-2}}\sum_{k=1}^{\infty}P(Z>\dnn \sqrt{k})\overset{N\to\infty}\longrightarrow C
\qwe
for some constant $C>0$. Since $P(Z\geq \dnn\sqrt{k})$ is decreasing in $k$ we have the bounds
\qw
\int_{1}^{\infty}P(Z\geq \dnn\sqrt{x})dx\leq \sum_{k=1}^{\infty}P(Z>\dnn \sqrt{k}) \leq \int_{1}^{\infty}P(Z\geq \dnn\sqrt{x})dx +1.
\qwe 
Hence, it is enough to prove that
\qw\label{burda}
\frac{1}{\an^{-2}}\int_{1}^{\infty}P(Z\geq \dnn\sqrt{x})dx \overset{N\to\infty}\longrightarrow C
\qwe
for some constant $C>0$. By substitution $\sqrt{x}=y$ we have
\qw
\int_{1}^{\infty}P(Z\geq \dnn\sqrt{x})dx=\int_{1}^{\infty}P(Z\geq \dnn y)2ydy=\int_1^\infty \int_{\dnn y}^{\infty}\frac{e^{-z^2/2}}{\sqrt{2\pi}}dz2ydy
\qwe
By switching the order of integration the last term above is equal to
\qw
\int_{\dnn}^\infty \int_1^{z/\dnn}\frac{e^{-z^2/2}}{\sqrt{2\pi}}2ydydz=\int_{\dnn}^\infty \frac{e^{-z^2/2}}{\sqrt{2\pi}}(z^2/\dnn^2-1)dz.
\qwe
This finishes the proof of (\ref{burda}), and consequently the proof of part (ii). 
\end{proof}

We begin the proof of (\ref{emel}) by rewriting $\frac{1}{\nu\antk}\sum_{k=1}^{\nu}\pro(\taun\geq k)\pro(\taun\geq \nu-k)$. Using $\pro(\taun\geq k)=\pro(k\leq \taun < \infty)+\pro(\taun=\infty)$ this expression can be written as
\begin{align}
\label{mother}\frac{\pro(\taun=\infty)^2}{\antk}+\frac{\pro(\taun=\infty)}{\nu\antk}&\sum_{k=1}^{\nu}\left(
\pro(k\leq\taun<\infty)+\pro(\nu-k\leq\taun<\infty)\right)\\\nonumber+&\frac{1}{\nu\antk}\sum_{k=1}^{\nu}\pro(k\leq\taun<\infty)\pro(\nu-k\leq\taun<\infty).
\end{align}

By part (i) of Lemma \ref{ilkfirs} we know that the first term in (\ref{mother}) converges to $K:=K_1^2=2\beta^{-2}p$.

By part(ii) of Lemma \ref{ilkfirs} we have for some positive constant $C$
\qw
\sum_{k=1}^{\nu}\left(
\pro(k\leq\taun<\infty)+\pro(\nu-k\leq\taun<\infty)\right)\leq 2 \ort[\taun,\;\taun<\infty]\leq \frac{C}{\an}
\qwe
for all $N$ large enough. Using once again part (i) of Lemma \ref{ilkfirs}, the second term in (\ref{mother}) is bounded above $\frac{C}{\nu\antk}$. However, since $\nu\antk\overset{N\to\infty}\longrightarrow \infty$ this term converges to 0 as $N$ diverges.

We partition the sum in the second term into two: $k=1,\dots,\lfloor \nu/2 \rfloor$ and $k=\lceil \nu/2 \rceil,\dots,\nu$. We have
\qw\label{durak}
\frac{1}{\nu\antk}\sum_{k=1}^{\lfloor \nu/2 \rfloor}\pro(k\leq\taun<\infty)\pro(\nu-k\leq\taun<\infty)\leq \frac{1}{\nu\antk} E[\taun;\;\taun<\infty] P(\nu/2\leq \taun<\infty).
\qwe
By Cheybshev Inequality and part (ii) of Lemma \ref{ilkfirs} we have
\qw
P(\nu/2\leq \taun<\infty)\leq \frac{2E[\taun;\;\taun<\infty]}{\nu}\leq \frac{C}{\nu\an}
\qwe
for $N$ large enough. Hence, (\ref{durak}) is bounded above by $C/\nu^2\an^4$ which converges to $0$ with $N$. The estimate of the second partition can be done similarly. Thus, we get
\qw
\frac{1}{\nu\antk}\sum_{k=1}^{\nu}\pro(\taun\geq k)\pro(\taun\geq \nu-k)\overset{N\to\infty}\longrightarrow 2\beta^{-2}p.
\qwe
This finishes the proof of (\ref{emel}) and hence, the proof of Proposition \ref{esasli}.

\end{proof}
\newcommand{\alb}{\frac{\alpha_N}{\beta}}
\begin{proof}[Proof of Proposition \ref{pasa}]
Using the method introduced at the beginning of the proof of Proposition \ref{esasli} and the terminology within, $\frac{\rn}{\nu}\mathbb{E}[1-\exp(-\rho\antk\kac)]$ is equal to
\qw\label{kopar}
\frac{\rn}{\nu}\int_{\mathbb{R}^{\nu}}\frac{\text{d}z e^{-\frac{1}{2}\sum_{i=1}^{\nu}z_i^2}}{(2\pi)^{\nu/2}}\left\{1-\exp(-\rho\antk\sum_{i=1}^{\nu}\mathbf{1}\{G_i(z)\geq C_N(x)\})\right\}.
\qwe
Let $D_k=\{z:G_k(z)>G_i(z)\;\forall i\not=k\}$. On $D_k$, we do the change of variables
\begin{equation}
 \begin{array}{cc} z_i=b_i+\Gamma_i \alb\sqrt{N}&\text{if }i\leq k,\\ z_i=b_i-\Gamma_i \alb\sqrt{N}&\text{if }i>k.\end{array}
\end{equation}
Then (\ref{kopar}) becomes
\begin{align}
 \label{hadicanim}&\frac{\beta^{-1}\sqrt{2\pi N}}{\an\nu x}\sum_{k=1}^{\nu}\int_{D_k'}\frac{\text{d}be^{-\frac{1}{2}\sum_{i=1}^\nu b_i^2}}{(2\pi)^{\nu/2}}\exp(-G_k(b)\alb \sqrt{N})\times\\&\nonumber\left\{1-\exp\left(-\rho\antk\sum_{i=1}^{\nu}\mathbf{1}\left\{G_k(b)\geq 2\sqrt{\frac{p}{N}}\left[\sum_{i+1}^kb_j+|i-k|\alb\sqrt{p}+\frac{\log x}{2\an\beta\sqrt{p}}\right]\right\}\right)\right\},
\end{align}
where
\begin{equation}
 D_k'=\{b:\;\sum_{i+1}^kb_j>-|i-k|\alb\sqrt{p}\;\forall i\not=k\}.
\end{equation}
On $D_k'$ we do the change of variables $a_1=\an\beta\sqrt{N}G_k(b)$ and $a_i=b_i$ for $i\geq 2$.
Denote by $D_k''$ the image of $D_k'$ under this change variables.
Since the curly bracket term above is always less than 1, by the exact same way in the proof of Proposition \ref{esasli} ((\ref{anane}) and the paragraph following it), up to an exponentially small error the above integral is equal to
\begin{align}
 \label{gayret}&\frac{\beta^{-2}}{\antk \nu }\sum_{k=1}^{\nu}\int_{D_k''}\frac{\text{d}a}{(2\pi)^{(\nu-1)/2}}e^{-\frac{1}{2}\sum_{i=2}^\nu a_i^2} \exp(-\frac{a_1}{\beta^2}-\frac{a_1^2}{2\Gamma_1^2\antk\beta^2 N})\times\\&
\nonumber\left\{1-\exp\left(-\rho\antk\sum_{i=1}^{\nu}\mathbf{1}\left\{a_1\geq 2\an\beta\sqrt{p}\left[\sum_{i+1}^k a_j+|i-k|\alb\sqrt{p}\right]+\log x\right\}\right)\right\}.
\end{align}
Note that since on $D_k''$ we have $\sum_{j=i+1}^k a_j+|i-k|\alb\sqrt{p}\geq0$ for all $i$, if $a_1<\log x$ the inner curly bracket term above is zero for all $i$. Consequently, if $a_1<\log x$ the integral in (\ref{gayret}) is equal to zero. Using
this and the fact that $D_k''$ does not depend on the first coordinate we can restrict the domain of integration of the above integral to $[\log x,\infty]\times \bar{D}_k''$ where $\bar{D}_k''$ is the projection of $D_k''$ to the last $\nu-1$ coordinates.
If we do the change variables $a_1-\log x= y$, up to a small error that vanishes as $N$ goes to infinity uniformly in $x$ on compact subsets of $(0,\infty)$, (\ref{gayret}) is equal to
\begin{align}
 \label{teyze}&\frac{\beta^{-2}}{\antk \nu x^{1/\beta^2}\Gamma_1}\sum_{k=1}^{\nu}\int_{0}^\infty\text{d}y\int_{\bar{D}_k''}\frac{\text{d}\bar{a}}{(2\pi)^{(\nu-1)/2}}e^{-\frac{1}{2}\sum_{i=2}^\nu a_i^2} \exp(-\frac{y}{\beta^2}-\frac{y^2}{2\Gamma_1^2\antk\beta^2 N})\times\\&
\nonumber\left\{1-\exp\left(-\rho\antk\sum_{i=1}^{\nu}\mathbf{1}\left\{y\geq 2\an\beta\sqrt{p}\left[\sum_{i+1}^k a_j+|i-k|\alb\sqrt{p}\right]\right\}\right)\right\},
\end{align}
where $\bar{a}$ is the projection of $a$ to the last $\nu-1$ coordinates. Now we work on
\qw\label{comeon}
 \int_{\bar{D}_k''}\frac{\text{d}\bar{a}e^{-\frac{1}{2}\sum_{i=2}^\nu a_i^2}}{(2\pi)^{(\nu-1)/2}}\left\{1-\exp\left(-\rho\antk\sum_{i=1}^{\nu}\mathbf{1}\left\{y\geq 2\an\beta\sqrt{p}\left[\sum_{i+1}^k a_j+|i-k|\alb\sqrt{p}\right]\right\}\right)\right\}.
\qwe
Let $W=(W_2,\dots,W_{\nu})$ be a sequence of i.i.d. standard normal random variables. Then, (\ref{comeon}) is equal to
\qw\label{babane}
 \mathbb{P}(W\in\bar{D}_k'')\mathbb{E}[1-\exp\left(-\rho\antk\sum_{i=1}^{\nu}\mathbf{1}\left\{y\geq 2\an\beta\sqrt{p}\left[\sum_{i+1}^k W_j+|i-k|\alb\sqrt{p}\right]\right\}\right)| W\in\bar{D}_k''].
\qwe
Note that the expectation in (\ref{babane}) is always between 0 and 1. Since on $\bar{D}_k''$ we have $\sum_{i+1}^k a_j+|i-k|\an\beta^{-1}\sqrt{p}>0$ it follows that when $y\sim 0$ the argument of the exponential in the expectation in (\ref{babane}) is close to zero. In other words, as $y\to 0^+$ we have
\begin{align}
\nonumber &\mathbb{E}[1-\exp\left(-\rho\antk\sum_{i=1}^{\nu}\mathbf{1}\left\{y\geq 2\an\beta\sqrt{p}\left[\sum_{i+1}^k W_j+|i-k|\alb\sqrt{p}\right]\right\}\right)| W\in\bar{D}_k'']\\\label{katki}&\sim \rho\antk\sum_{i=1}^{\nu}\mathbb{E}[\mathbf{1}\left\{y\geq 2\an\beta\sqrt{p}\left[\sum_{i+1}^k W_j+|i-k|\alb\sqrt{p}\right]\right\}| W\in\bar{D}_k'']\\\nonumber&\sim
\rho\antk\sum_{i=1}^{\nu}P(y\geq 2\an\beta\sqrt{p}(R_{k-i}+|i-k|\alb\sqrt{p})|R_{k-i}>-\an\beta^{-1}\sqrt{p}|i-k|),
\end{align}
where $R_{k-i}$ is a centered normal random variable with variance $|k-i|$. The probability term on the last display is equal to
\begin{align*}
 &P\left(Z\leq \frac{y}{2\an\beta\sqrt{p}\sqrt{|k-i|}}-\alb\sqrt{p}\sqrt{|i-k|}\left|Z>-\an\beta^{-1}\sqrt{p}\sqrt{|i-k|}\right.\right),
\end{align*}
where $Z$ is a standard normal random variable. Note that the above term converges to 0 at least exponentially if $\an\sqrt{|k-i|}\gg 1$. Hence, the contribution from such $i$ to the sum in (\ref{katki}) is negligible. If $\an\sqrt{|k-i|}\ll 1$ the above term converges to 1. The number of such $i$'s is $o(1/\antk)$. Hence, the contribution from these $i$'s to the sum in (\ref{katki}) is also negligible. Finally, if $\an\sqrt{|k-i|}=c$ the above term is equal to
\begin{equation}
 \frac{P(-c'\leq Z\leq \frac{y}{2c'}-c')}{P(-c'\leq Z)}\overset{y\to0+}\sim c'' y.
\end{equation}
Hence, for some positive constants $c_1,c_2$ independent of $k$ we have for all $y>0$
\begin{equation}\label{jumbo}
 (1\wedge \rho c_1 y)P(W\in\bar{D}_k'') \leq (\ref{comeon}) \leq (1\wedge \rho c_2 y)P(W\in\bar{D}_k'') .
\end{equation}
Hence, the integral in (\ref{teyze}) is bounded below and above by
\begin{equation}\label{sukur}
 \frac{1}{x^{1/\beta^2}}\left\{\frac{1}{\antk\nu}\sum_{k=1}^{\nu}P(W\in\bar{D}_k'')\right\}\left\{\frac{1}{\beta^2}\int_{0}^{\infty}(1\wedge c\rho y)\exp(-\frac{y}{\beta^{2}})\text{d}y\right\},
\end{equation}
with different constants $c$, for $N$ large enough. After a simple change of variables the second curly bracket term above is equal to $\int_0^\infty (1\wedge c\rho y)e^{-y}\text{d}y$ with $c=c/\beta^2$. 
Note that with the notation of the proof of Proposition \ref{esasli}
\begin{equation}
 \frac{1}{\antk\nu}\sum_{k=1}^{\nu}P(W\in\bar{D}_k'')=\frac{1}{\antk \nu}\sum_{k=1}^{\nu}\pro(\tau_N\geq k)\pro(\tau_N\geq \nu-k)\overset{N\to\infty}\longrightarrow K,
\end{equation}
where $K$ is as in the statement of Proposition \ref{esasli}. Thus,  
\begin{equation}
\frac{K}{x^{1/\beta^2}}\int_{0}^{\infty}(1\wedge \rho c_1 y)e^{-y}dy\leq (\ref{teyze}) \leq \frac{K}{x^{1/\beta^2}}\int_{0}^{\infty}(1\wedge \rho c_2 y)e^{-y}dy,
\end{equation}
for some positive constants $c_1$ and $c_2$. This finishes the first part of Proposition \ref{pasa} with
\begin{equation}\label{adana}
 C_i(\rho)=\int_{0}^{\infty}(1\wedge \rho c_i y)e^{-y}dy,\;\;i=1,2.
\end{equation}
Moreover, we have for any $c>0$
\begin{equation}
 C(\rho)=\int_{0}^{\infty}(1\wedge \rho c y)e^{-y}dy=\int_{0}^{\frac{1}{c\rho}}c\rho y e^{-y}dy+\int_{\frac{1}{c\rho}}^{\infty}e^{-y}= C\frac{1}{\rho}+e^{-\frac{1}{c\rho}}\underset{\rho\to\infty}\longrightarrow 1.
\end{equation}
This proves the second claim of Proposition \ref{pasa}.

\end{proof}

\section{Comparison}
\newcommand{\rijs}{\Lambda_{ij}^{0}}
\newcommand{\rijb}{\Lambda_{ij}^{1}}
\newcommand{\rijh}{\Lambda_{ij}^{h}}
\newcommand{\xni}{X_N^0(i)}
\newcommand{\xnni}{X_N^1(i)}
In this section we compare the extremal statistics of the original Gaussian Hamiltonians of the correlated mean field models with the block independent Gaussian processes described in the previous sections. 
Recall that given a realization of the SRW, $Y_N$, the Hamiltonians of the SK and the $p$-spin models are given by a Gaussian processes $\xni=H_N(Y_N(i))$ where $\xni$ is a centered Gaussian process with the covariance structure
\begin{equation}
\Lambda_{ij}^{0}=\mathbb{E}[\xni X_N^0(j)]=\left(1-\frac{2\text{dist}(Y_N(i),Y_N(j))}{N}\right)^p,\hspace{0.3in}p\geq 2.
\end{equation} 
Also recall that by $X_N^1$ we denote the auxiliary Gaussian process that we will use to approximate the extremal statistics of $\xni$. $X_N^1(i),\;i\in\mathbb{N}$ is a Gaussian process with covariance matrix
\begin{equation}
 \Lambda_{ij}^1=\mathbb{E}[\xnni X_N^1(j)]=\left\{\begin{array}{cc}1-2p|i-j|/N &\text{ if } \lfloor i/ \vv\rfloor=\lfloor j/\vv\rfloor, \\ 0 &\text{    otherwise}.
\end{array}\right.
\end{equation}
Recall that $w_{\rho}$ is a random subset of $\mathbb{N}$ where $\mathbb{P}(i\in w_{\rho})=\rho\antk$, i.i.d for $i\in\mathbb{N}$, and $\mathcal{W}$ denotes its $\sigma$-algebra. Finally, recall that the time scales we are considering are of the form $\tn=\exp(\an N)$ where $\an=N^{-c},\;c\in(0,1/2)$.

\begin{proposition}\label{berker}
Fix sequences $\{t_k\}$ and $\{x_k\}$ i.e. $0\leq t_1\leq t_2\leq \dots \leq t_l=t$ and $0< x_1 \leq x_2\leq \dots\leq x_l$,

\vspace{0.1in}
{\bf (i)} For $p\geq 2$ for any $c\in(0,1/2)$, $\mathcal{Y}$ a.s.
\begin{align}
 \lim_{N\to \infty}\big|\mathbb{P}&(\max_{i\leq t_1 \rn,i\in w_{\rho}} {\xni}\leq C_N(x_1),\dots,\max_{i\leq t_l \rn,i\in w_{\rho}}{\xni}\leq C_N(x_l)|\mathcal{Y})-
\\ \nonumber &\mathbb{P}(\max_{i\leq t_1 \rn,i\in w_{\rho}} {\xnni}\leq C_N(x_1),\dots,\max_{i\leq t_l \rn,i\in w_{\rho}}{\xnni}\leq C_N(x_l)\big|=0.
\end{align}

{\bf (ii)} For $p=2,\;p\geq4$ for any $c\in(0,1/2)$, and, for $p=3$ for any $c\in(0,1/4)$, $\mathcal{Y}$ a.s. 
\begin{align}\label{basak}
 \liminf_{N\to \infty}&\left\{\mathbb{P}(\max_{i\leq t_1 \rn} {\xni}\leq C_N(x_1),\dots,\max_{i\leq t_l \rn}{\xni}\leq C_N(x_l)|\mathcal{Y})-\right.\\ \nonumber &\left.\mathbb{P}(\max_{i\leq t_1 \rn} {\xnni}\leq C_N(x_1),\dots,\max_{i\leq t_l \rn}{\xnni}\leq C_N(x_l))\right\}=0.
\end{align}
\end{proposition}

The result of the first part of Proposition \ref{berker} is that the extremal distributions of $X_N^0$ and $X_N^1$ are comparable on the diluted random subset of indices $w_{\rho}$. The second part is needed needed for to extend this comparison
to the whole set of indices; that's where we need stronger restriction on $\an$ for $p=3$.

To prove Proposition \ref{berker} we use the well-known interpolation estimate for Gaussian processes.
\begin{theorem}\label{lead} (Normal Comparison Lemma, Theorem 4.2.1 on page 81 in \cite{lead})
 Suppose $\xi_1,\dots,\xi_n$ are standard normal variables with covariance matrix $\Lambda^1=(\Lambda_{ij}^1)$ and $\mu_1,\dots,\mu_n$ similarly with covariance matrix $\Lambda^0=(\Lambda_{ij}^0)$ and $u_i\in \mathbb{R}$. Let $\Lambda^h_{ij}:=h\Lambda_{ij}^1+(1-h)\Lambda_{ij}^0$ then
\begin{align}
 \nonumber\mathbb{P}&(\xi_i\leq u_i:\; i=1,\dots,n)-\mathbb{P}(\mu_i\leq u_i:\; i=1,\dots,n)
\\&\leq \frac{1}{2\pi}\sum_{1\leq i<j\leq n}(\Lambda_{ij}^1-\Lambda_{ij}^0)_+ \int_0^1(1-(\Lambda_{ij}^h)^2)^{-1/2}\exp(-\frac{u_i^2+u_j^2-2\Lambda_{ij}^h u_i u_j}{2(1-(\Lambda_{ij}^h)^2)})\text{d}h.
\end{align}
\end{theorem}
\begin{proof}[Proof of Proposition \ref{berker} part (i)]

Let $\rijh=h\rijs+(1-h)\rijb$. Let $l(i)$ and $l(j)$ be such that $t_{l(i)-1}\rn<i\leq t_{l(i)}\rn$ and $t_{l(j)-1}\rn<j \leq t_{l(j)}\rn$. 
Then we use Theorem \ref{lead} with $u_i=\an\beta^{-1}\sqrt{N}+\frac{\log(x_{l(i)})}{\an\beta\sqrt{N}}$. Note that it is always the case that $\Lambda_{ij}^{h}\leq (\Lambda_{ij}^{0})_+$. Then it is not hard to see that for any sequences $\{t_k\}$ and $\{x_k\}$ we can find a constant $C$ s.t.
uniform in $h\in[0,1]$ for $N$ large enough
\begin{equation*}
\exp(-\frac{u_i^2+u_j^2-2\Lambda_{ij}^h u_i u_j}{2(1-(\Lambda_{ij}^{h})^2)})\leq C\exp(-\frac{\antk\beta^{-2} N}{1+(\Lambda_{ij}^{0})_+}).
\end{equation*}
Hence we have $\mathcal{Y}$ and $\mathcal{W}$ a.s.
\begin{align*}
 &\left|\mathbb{P}(\max_{i\leq t_1 \rn,i\in w_{\rho}}\right. {\xni}\leq C_N(x_1),\dots,\max_{i\leq t_l \rn,i\in w_{\rho}}{\xni}\leq C_N(x_l)|\mathcal{Y},\mathcal{W})-\\ 
 &\qquad\mathbb{P}(\max_{i\leq t_1 \rn,i\in w_{\rho}} {\xnni}\leq C_N(x_1),\dots,\left.\max_{i\leq t_l \rn,i\in w_{\rho}}{\xnni}\leq C_N(x_l)|\mathcal{W})\right|\leq\\ &
C \sum_{i< j}^{t\rn}\mathbf{1}\{w_{N,\rho}(i)=w_{N,\rho}(j)=1\}|\rijs-\rijb|\exp(-\frac{\antk \beta^{-2} N}{1+(\Lambda_{ij}^{0})_+})\int_0^1 (1-(\Lambda_{ij}^{h})^2)^{-1/2}\text{d}h.
\end{align*}
\newcommand{\lz}{\rijs}
\newcommand{\lo}{\rijb}
Since $\mathbb{P}(w_{N,\rho}(i)=w_{N,\rho}(j)=1)=\rho^2\an^4$ we get
\begin{align}
 \nonumber&\left|\mathbb{P}(\max_{i\leq t_1 \rn,i\in w_{\rho}}\right. {\xni}\leq C_N(x_1),\dots,\max_{i\leq t_l \rn,i\in w_{\rho}}{\xni}\leq C_N(x_l)|\mathcal{Y})-\\ 
\nonumber &\qquad\mathbb{P}(\max_{i\leq t_1 \rn,i\in w_{\rho}} {\xnni}\leq C_N(x_1),\dots,\left.\max_{i\leq t_l \rn,i\in w_{\rho}}{\xnni}\leq C_N(x_l))\right|\leq\\ &\label{lazim}
\qquad C\rho^2\an^4 \sum_{i< j}^{t\rn}|\lz-\lo|\exp(-\frac{\antk \beta^{-2} N}{1+(\Lambda_{ij}^{0})_+})\int_0^1 (1-(\Lambda_{ij}^{h})^2)^{-1/2}\text{d}h.
\end{align}
If $\lfloor i/\nu \rfloor = \lfloor j/\nu \rfloor$ then dist$(Y_N(i),Y_N(j))\leq |i-j|$, and as a consequence, $\lz\geq \lo>0$. Hence, $\int_0^1 (1-(\rijh)^2)^{-1/2}\text{d}h\leq (1-(\rijs)^2)^{-1/2} $. If $\lfloor i/\nu \rfloor \not= \lfloor j/\nu\rfloor$ then $\rijb=0$ and $\int_0^1 (1-(\rijh)^2)^{-1/2}\text{d}h\leq C$. Hence, (\ref{lazim}) is bounded above by
\begin{align}
\label{lazimiki}C\rho^2\an^4 &\left\{\sum_{\overset{i< j}{\lfloor i/\nu \rfloor =\lfloor j/\nu \rfloor}}^{t\rn}|\lz-\lo|(1-(\rijs)^2)^{-1/2}\exp(-\frac{\antk \beta^{-2} N}{1+(\Lambda_{ij}^{0})_+})\right.\\ \nonumber
&+ \left.\sum_{\lfloor i/\nu \rfloor \not= \lfloor j/\nu \rfloor}^{t\rn}|\lz|\exp(-\frac{\antk \beta^{-2} N}{1+(\Lambda_{ij}^{0})_+})\right\}.
\end{align}
Let us define
\newcommand{\nur}{\nu_{\rho}}
\newcommand{\mur}{\mu_{\rho}}
\newcommand{\rur}{r(N)}
\newcommand{\lf}{\lfloor}
\newcommand{\lr}{\rfloor}
\newcommand{\rzd}{\Lambda_{d}^{0}}
\newcommand{\ydij}{D_{ij}}
\newcommand{\dd}{||d||}
\begin{equation*}
 D_{ij}=\text{dist}(Y_N(i),Y_N(j)),\hspace{0.4in}\Lambda_{d}^{0}=\left(1-\frac{2d}{N}\right)^p.
\end{equation*}
(\ref{lazimiki}) is bounded above by
\begin{align}
 \nonumber C\an^4\sum_{d=0}^N&\left\{\sum_{\overset{i<j}{\lfloor i/\nu\rfloor=\lfloor j/\nu\rfloor}}^{t\rn}(\rzd-\lo)\mathbf{1}\{\ydij=d\}(1-(\rzd)^2)^{-1/2}\exp(-\frac{\antk\beta^{-2} N}{1+\rzd})\right.\\\label{sudan}
&+\sum_{\lfloor i/\nu\rfloor\not=\lfloor j/\nu\rfloor}^{t\rur}(\rzd)_+\mathbf{1}\{\ydij=d\}\exp(-\frac{\antk\beta^{-2} N}{1+(\rzd)_+})\\\nonumber
&+\left.\sum_{\lfloor i/\nu\rfloor\not=\lfloor j/\nu\rfloor}^{t\rur}(\rzd)_-\mathbf{1}\{\ydij=d\}\exp(-\antk\beta^{-2} N)\right\}.
\end{align}
We need the following lemma which will be proved in the next section.
\begin{lemma}\label{aslanimik}
 Let $||d||=\min (d,N-d)$. For any $\eta>0$, there exists a constant, $C=C(\nu,\eta,c)$ such that, $\mathcal{Y}$-a.s. for $N$ large enough, for all $d \in \{0,\dots ,N\}$
\begin{equation}\label{cimbomik}
 \sum_{\lfloor i/\nu\rfloor\not=\lfloor j/\nu\rfloor}^{t\rur}\mathbf{1}\{\ydij=d\}\leq C\left\{t^2 \rur^2 2^{-N}\comnd+t\frac{\rur e^{\eta\antk||d||}}{\nu\antk}\right\},
\end{equation}
and
\begin{equation}\label{cimbombir}
 \sum_{\overset{i<j}{\lfloor i/\nu\rfloor=\lfloor j/\nu\rfloor}}^{t\rur}\mathbf{1}\{\ydij=d\}(\rzd-\rijb)\leq C\rur\frac{d^2}{N^2} \mathbf{1}\{d\leq \nu\}.
\end{equation}
\end{lemma}
\vspace{0.1in}

By Lemma \ref{aslanimik}, the first line of (\ref{sudan}) is bounded above by
\begin{equation}\label{estone}
 C\an^4\sum_{d=0}^{\nu}  \frac{\rur d^2}{N^2} (1-(\rzd)^2)^{-1/2}\exp(-\frac{\antk\beta^{-2} N}{1+\rzd}).
\end{equation}
The second line of (\ref{sudan}) is bounded above by the sum of
\begin{equation}\label{esttwo}
 C\an^4\sum_{d=0}^N t^2{\rur^2}2^{-N}\comnd (\rzd)_+\exp(-\frac{\antk\beta^{-2} N}{1+(\rzd)_+})
\end{equation}
and
\begin{equation}\label{estthree}
 C\antk\sum_{d=0}^N \frac{\rur}{\nu}e^{\eta\antk||d||}(\rzd)_+\exp(-\frac{\antk\beta^{-2} N}{1+(\rzd)_+}).
\end{equation}
Finally, the third line of (\ref{sudan}) is bounded above by
\begin{equation}\label{estfour}
 C\an^4\sum_{d=N/2}^N \left\{t^2{\rur^2}2^{-N}\comnd+t\frac{\rur}{\nu\antk}e^{\eta\antk||d||}\right\}(\rzd)_- \exp(-\antk\beta^{-2} N).
\end{equation}
We start working on the estimate of (\ref{esttwo}). Let $I(u)$ be
\begin{equation}\label{hadkis}
 I(u)=u\log u +(1-u) \log(1-u)+\log 2,
\end{equation}
and let $J_N(u)$ be
\begin{equation}\label{berrat}
 J_N(u):=2^{-N}{{N}\choose{\lfloor Nu \rfloor}}e^{NI(u)}\sqrt{\frac{\pi N}{2}}.
\end{equation}
Using Sterling's formula we have $J_N(u)\overset{N\to\infty}\longrightarrow(4u(1-u))^{-1}$ uniform in $u$ on compact subsets of $(0,1)$. Also, there exists a constant $c$ s.t. $J_N(u)\leq cN^{1/2}$ for all $N$ and for all $u\in[0,1]$. Hence, using the definition of $\rn$, (\ref{esttwo}) is bounded above by
\newcommand{\yup}{\Upsilon_{N,p,\beta}(u)}
\newcommand{\ecnr}{\antk\beta^{-2} N}
\begin{align}
\label{esttwotwo}
C \antk N^{1/2}\sum_{d=0}^N J_N\left(\frac{d}{N}\right)\left(1-\frac{2d}{N}\right)^p_+\exp\left\{N\Upsilon_{N,p,\beta}\left(\frac{d}{N}\right)\right\},
\end{align}
where
\begin{equation}
 \yup:=\antk\beta^{-2}-\frac{\antk\beta^{-2}}{1+(1-2u)^p_+}-I(u).
\end{equation}

Since $\an\overset{N\to\infty}\longrightarrow 0$, by (\ref{hadkis}) it is easy to see that for all $p\geq 2$ there exist positive constants $\delta,\delta'$ and $c$ s.t.
\begin{align}
 \qquad \qquad&\yup\leq -\delta' \hspace{0.2in}\hspace{0.58in}\text{for all }u\in[0,1]\setminus (1/2-\delta,1/2+\delta),
 \\\label{guzelyaz}&\yup\leq -c(u-1/2)^2\hspace{0.2in}\text{for all }u\in(1/2-\delta,1/2+\delta).
\end{align}
Then, the sum over $d$'s such that $d/N\notin (1/2-\delta,1/2+\delta)$ in (\ref{esttwotwo}) is bounded above by $\exp(-\delta''N)$ for some $\delta''>0$ small enough.

Now we estimate the sum in (\ref{esttwotwo}) over $d$'s with $d/N\in(1/2-\delta,1/2+\delta)$. Note that $J_N(d/N)\leq c$ uniformly for $d\in(1/2-\delta,1/2+\delta)$. Using this and (\ref{guzelyaz}) 
the sum over such $d$'s in (\ref{esttwotwo}) is bounded above by
\begin{align*}
&
\leq C \antk N^{1/2}\sum_{d/N\in (1/2-\delta,1/2+\delta)}\left(1-\frac{2d}{N}\right)^p_+\exp\left\{-cN\left(\frac{d}{N}-\frac{1}{2}\right)^2\right\}\\&
\leq C\antk N^{3/2}\int_{1/2-\delta}^{1/2+\delta}(1-2u)_+^p\exp(-c'N(u-1/2)^2)du\\&
\leq C\antk N^{3/2}\int_{0}^{\delta}x^p\exp(-c'N x^2)dx\\&
\leq C\antk N^{3/2}\int_{0}^{\delta {N^{1/2}}}\frac{y^p}{N^{p/2}}\exp(-c'y^2)\frac{dy}{N^{1/2}}\\&
\leq C\antk N^{1-p/2}\overset{N\to\infty}\longrightarrow 0,
\end{align*}
since $\antk \to 0 $ as $N\to\infty$ and $p\geq 2$. This finishes the estimate on (\ref{esttwo}).

Now we work on (\ref{estthree}). (\ref{estthree}) is bounded above by
\newcommand{\yuptil}{\tilde{\Upsilon}_{N,p,\beta}}
\newcommand{\yuptilndn}{\tilde{\Upsilon}_{N,p,\beta}(d/N)}
\qw\label{estthreethree}
\leq \frac{C\an N^{1/2}}{\nu}\sum_{d=0}^N \left(1-\frac{2d}{N}\right)^p_+\exp\left\{N\yuptil\left(\frac{d}{N}\right)\right\},
\qwe
where
\begin{equation*}
 \yuptil(u)=\frac{\antk\beta^{-2}}{2}+\eta\antk||u||-\frac{\antk\beta^{-2}}{1+(1-2u)_+^p},
\end{equation*}
and $||u||:=\min(u,1-u)$. It is clear that for any $p\geq 2$, for $\eta$ small enough we can find positive constants $\delta,\delta'$ and $c$ s.t. for all $N$ large enough
\begin{equation}
  \yuptil\leq -\antk\delta' \hspace{0.2in}\text{for all }u\in [\delta,1-\delta],
\end{equation}
and
\begin{align}
 &\yuptil\leq -c\antk u\hspace{0.527in}\text{for all }u\in [0,\delta],\\&
\yuptil\leq -c\antk (1-u)\hspace{0.2in}\text{for all }u\in [1-\delta,1].
\end{align}
By this and the fact that $\antk N\overset{N\to\infty}\longrightarrow\infty$ the sum in (\ref{estthree}) over $d$'s such that $d/N\in[\delta,1-\delta]$ is bounded above by $\exp(-\delta''\antk N)$ for $\delta''>0$ and 
hence, does not pose a problem. 
The sum over $d$'s with $d/N\in[0,\delta]$ in (\ref{estthreethree}) is bounded above by
\begin{align*}
 &\hspace{0.2in}\frac{C\an N^{1/2}}{\nu}\sum_{d/N\in[0,\delta]}\left(1-\frac{2d}{N}\right)^p_+\exp\left\{-\delta'\antk N\left(\frac{d}{N}\right)\right\}\\
&\leq \frac{C\an N^{3/2}}{\nu}\int_0^{\delta}(1-2u)^p\exp(-c'\antk N u)du\\&
\leq \frac{C\an N^{3/2}}{\nu}\int_0^{\delta\antk N}\exp(-cx)\frac{dx}{\antk N}\\&
\leq C\frac{N^{1/2}}{\an \nu}\overset{N\to\infty}\longrightarrow 0,
\end{align*}
since we have $N^{1/2}\an^{-1}\ll \nu$ (recall (\ref{bukadar})). The estimate for the sum over $d$'s with $d/N\in[1-\delta,1]$ can be done analogously. Hence, the error term (\ref{estthree}) goes to 0 as $N\to\infty$.

Now we estimate (\ref{estone}). (\ref{estone}) is bounded above by
\begin{equation}
 \label{estoneone}C\an^3\sum_{d=0}^{\nu} N^{1/2}\frac{d^2}{N^2}\left(1-(1-2dN^{-1})^{p}\right)^{-1/2}\exp\left\{\antk\beta^{-2} N\left(\frac{1}{2}-\frac{1}{1+(1-2dN^{-1})^p}\right)\right\}
\end{equation}
Note that since $d\leq \nu$ and $\nu\ll N$ we can find constant $c_1,c_2$ such that for all $d\in\{1,\dots,\nu\}$
\begin{equation}
 1-c_1\frac{d}{N}\leq \left(1-\frac{2d}{N}\right)^p\leq 1-c_2\frac{d}{N},
\end{equation}
for $N$ large enough. As a consequence (\ref{estoneone}) is bounded above by
\begin{align}
 &C\sum_{d=0}^{\nu}\an^3 N^{1/2} \left(\frac{d}{N}\right)^{3/2}\exp\left\{-c\antk\beta^{-2} N\left(\frac{d}{N}\right)\right\}\\
&\leq C\an^3 N^{3/2}\int_{0}^{\nu/N}u^{3/2}\exp(-c\antk\beta^{-2} N u)du\\
&\leq C\an^3 N^{3/2}\int_{0}^{\nu\antk}\frac{x^{3/2}}{\an^{3}N^{3/2}}\exp(-cx)\frac{dx}{\antk N}\\
&\leq C\frac{1}{\an^2 N}\overset{N\to\infty}\longrightarrow 0,
\end{align}
again since $\antk N\to\infty$ as $N$ diverges.

Finally, we work on (\ref{estfour}). Using (\ref{berrat}), the first term of (\ref{estfour}) is bounded above by
\begin{equation}\label{estfourfour}
 C\antk N^{1/2}\sum_{d=N/2}^N J_N\left(\frac{d}{N}\right)\exp\left\{-NI\left(\frac{d}{N}\right)\right\}\left(\frac{2d}{N}-1\right)^p.
\end{equation}
We can find constant $\delta,\delta'$ and $c$ such that
\begin{equation}\label{kararbir}
 I(u)\leq -c(u-1/2)^2\hspace{0.2in}\text{for all }u\in[1/2,1/2+\delta),
\end{equation}
and
\begin{equation}\label{karariki}
 I(u)\leq -\delta'\hspace{0.2in}\text{for all }u\in[1/2+\delta,1].
\end{equation}
Then, it can be shown that the first part of (\ref{estfour}) goes to 0 as $N\to\infty$ completely analogous to the proof of the estimate of (\ref{esttwo}).

Note that for $\eta$ small enough we can find a constant such that the second part of (\ref{estfourfour}) is bounded above by
\begin{align*}
 &\frac{\an N^{1/2}}{\nu}\sum_{d=N/2}^N \left(\frac{2d}{N}-1\right)^p\exp(-c\antk\beta^{-2} N)\leq \exp(-c'\antk N)
\end{align*}
for some $c'$ small enough, since $\antk N\to\infty$ as $N\to\infty$. This finishes the estimate of the second part of (\ref{estfour}), and thus, the proof of part $(i)$ Proposition \ref{berker}.
\end{proof}

\begin{proof}[Proof of Proposition \ref{berker} part (ii)]
To prove (\ref{basak}) we use Theorem \ref{lead} with $\xi=X_N^1$ and $\mu=X_N^0$. By the same arguments at beginning of the proof of part $(i)$ of Proposition \ref{berker} we have for some constant $C$
\begin{align}
 \nonumber&\left\{\mathbb{P}(\max_{i\leq t_1 \rn} {\xnni}\leq C_N(x_1),\dots,\max_{i\leq t_l \rn}{\xnni}\leq C_N(x_l))-\right.\\ \nonumber &\left.\mathbb{P}(\max_{i\leq t_1 \rn} {\xni}\leq C_N(x_1),\dots,\max_{i\leq t_l \rn}{\xni}\leq C_N(x_l)|\mathcal{Y})\right\}\leq\\\label{hulya}
&\qquad C\sum_{i<j}^{t\rn}(\Lambda_{ij}^1-\Lambda_{ij}^0)_+\exp(-\frac{\antk\beta^{-2} N}{1+(\Lambda_{ij}^0)_+})\int_0^1(1-(\Lambda_{ij}^h)^2)^{-1/2}\text{d}h.
\end{align}
\newcommand{\lf}{\lfloor}
\newcommand{\rf}{\rfloor}
\newcommand{\lz}{\Lambda_{ij}^0}
\newcommand{\lb}{\Lambda_{ij}^1}
As before, if $\lf i/\nu \rf =\lf j/\nu \rf$ then $\lb\leq \lz$ and subsequently $(\lb-\lz)_+=0$. If $\lf i/\nu \rf \not= \lf j/\nu \rf$ then $\lb=0$, then if $\lz<0$, $(\lb-\lz)_+=(\lz)_-$ and if $\lz>0$, $(\lb-\lz)_+=0$. Also, in this case, the integral term in the above display is bounded. Hence, (\ref{hulya}) is bounded above by
\begin{equation}\label{vildan}
 C\sum_{i<j}^{t\rn}(\Lambda_{ij}^0)_-\exp(-\antk\beta^{-2} N).
\end{equation}
Note that when $p\geq 2$ is even the above term is always zero and in this case (\ref{basak}) is trivial. From now on we assume that $p\geq 3$ is odd. Let $D_{ij}$ and $\Lambda_d^0$ be as before. Using the fact that $p$ is odd it is easy to see that (\ref{vildan}) is bounded above by
\begin{equation}\label{mehmet}
 C\sum_{d=N/2}^N\sum_{|i-j|>N/2}^{t\rn}\mathbf{1}\{D_{ij}=d\}\left(\frac{2d}{N}-1\right)^p\exp(-\antk\beta^{-2} N).
\end{equation}
Using the inequality (\ref{cimbomik}) of Lemma \ref{aslanimik} and the definition of $\rn$ we can see that (\ref{mehmet}) is bounded above by the sum of
\begin{equation}\label{hayat}
 C\sum_{d=N/2}^N \an^{-2} N  2^{-N}\comnd \left(\frac{2d}{N}-1\right)^p.
\end{equation}
and
\begin{equation}\label{yasam}
C\sum_{d=N/2}^N  \frac{N^{1/2}}{\nu \an^3}e^{\eta\antk (N-d)}\left(\frac{2d}{N}-1\right)^p\exp(-\antk\beta^{-2} N/2).
\end{equation}

We start with the estimate of (\ref{hayat}). Let $I(u)$ and $J_N(u)$ be as defined before. Using the properties of $I(u)$ and $J_N(u)$ we can see that (\ref{hayat}) is bounded above by the sum of
\qw\label{alala}
C\sum_{\frac{d}{N}\in [1/2,1/2+\delta)}\an^{-2}N^{1/2}\exp(-cN(d/N-1/2)^2)(2d/N-1)^p
\qwe
and
\qw\label{hocaaa}
C\sum_{\frac{d}{N}\in [1/2+\delta,1]}\an^{-2}N\exp(-N\delta' )(2d/N-1)^p,
\qwe
for some appropriate positive numbers $c,\delta$ and $\delta'$. It is too see that (\ref{hocaaa}) is exponentially small in $N$ and does not pose a problem. The sum in (\ref{alala}) is bounded above by a constant times
\qw
\an^{-2}N^{3/2}\int_{1/2}^{1/2+\delta}\exp(-cN(u-1/2)^2)(u-1/2)^p \text{d}u \leq C \an^{-2}N^{1-p/2},
\qwe
where for the last inequality we used the same changes of variables we used in the proof of Proposition \ref{berker}. Note that for $p\geq 4$, $\an^{-2}N^{1-p/2}$ converges to 0 with $N$ for any $c\in(0,1/2)$. It converges to 0 for $p=3$ as well 
if $\an^{-2}N^{1/2}=o(1)$ which is the case for $c\in(0,1/4)$ as in the hypothesis of part $(ii)$ of Proposition \ref{berker}. This finishes the estimate on (\ref{hayat}).

Finally, note that for $\eta$ small enough the sum in (\ref{yasam}) is exponential small in $N$. Hence, this finishes the proof of part $(ii)$ of Proposition \ref{berker}.
\end{proof}

\section{Random walk results}
\newcommand{\yn}{Y_N}
\newcommand{\kur}{K_N}
\newcommand{\dist}{\text{dist}}
\newcommand{\ydij}{D_{ij}}
\newcommand{\rur}{\rn}
In this section we prove Lemma \ref{aslanimik}. Let $\pro_x$ denote the probability law of the simple random walk $\yn$ started at $\yn(0)=x$. Let $Q=Q_k,\;k\in\mathbb{N}$ be a birth-death process on $\{1,\dots,N\}$ with transition probabilities $p_{k,k-1}=1-p_{k,k+1}=k/N$. Let $P_i$ and $E_i$ denote the law and expectation of $Q$ conditioned on $Q_0=i$. Let us also define $p_k(d)$ as $p_k(d):=P_0(Q_k=d)$. Note that, under $P_0$ for any $j\in\mathbb{N}$ we have
$\dist(\yn(0),\yn(k))\overset{\text{d}}=\dist(\yn(j),\yn(j+k))\overset{\text{d}}=Q_k.
$ Finally, let $T_d=\min\{k\geq 1:Q_k=d\}$, be the hitting time of $d$. 

A simple calculation shows that the weight of the invariant distribution of $Q$ at $d$ is equal to $2^{-N}\comnd$. The following theorem gives a sharp estimate for the difference of $p_k(\cdot)$ and the invariant measure, for $k$ large. It is stated and proved in \cite{bbc} using the coupling technique of \cite{matt} and we do not repeat it here.
\begin{theorem}\label{mathew}({\it Lemma 4.1 on page 17 in \cite{bbc}})

There exists a $K>0$ large enough such that for $k\geq K_N:=KN^2\log(N)$ for any $d\in\{0,1,\dots,N\}$
\begin{equation}
 \left|\frac{p_k(d)+p_k(d+1)}{2}-2^{-N}\comnd\right|\leq 2^{-4N}.
\end{equation}
\end{theorem}

\begin{lemma}\label{onceblok}
 Under the hypothesis of Lemma \ref{aslanimik} there exists a positive constant $C$ that does not depend on $d$, s.t. $\mathcal{Y}$ a.s.
\begin{equation}\label{fenerik}
 \sum_{\overset{i<j}{\lfloor i/\nu\rfloor=\lfloor j/\nu\rfloor}}^{t\rn}\mathbf{1}\{\ydij=d\}\leq Ct{\rn} \mathbf{1}\{d\leq \nu\},
\end{equation}
for all $d\in\{0,1,\dots,N\}$ and $N$ large enough.
\end{lemma}
\begin{proof}
Lemma is trivially true for $d>\nu$. Now we assume $d\leq \nu$. Define
\begin{equation}
 p(d)=E_0[\sum_{i=1}^{\nu}\mathbf{1}\{Q_{i}=d\}].
\end{equation}
Following the same arguments as in the first part of the proof of Lemma 4.2 in \cite{bbc} we have
\qw
p(d) \leq 2.
\qwe
Now we define the one-block contribution
\begin{equation}
 \sum_{i=1}^{\nu}\sum_{j=i+1}^\nu\mathbf{1}\{D_{ij}=d\}=:Z.
\end{equation}
Using the upper bound we have
\begin{equation}
 E[Z]=\sum_{i=1}^\nu E[\sum_{j=i+1}^\nu \mathbf{1}\{D_{ij}=d\}]\leq 2\nu.
\end{equation}
Since $Z\leq \nu^2$ a.s. 
\begin{equation}
 \text{Var}[Z]\leq \nu^4.
\end{equation}
The left-hand side of (\ref{fenerik}) is stochastically bounded above by $\sum_{k=1}^m Z_k$ where $m=\lceil\frac{t\rn}{\nu}\rceil$ and $Z_k$ is a sequence of i.i.d. copies of $Z$. Then using Chebyshev's inequality
\begin{align*}
 P(Z_1+\dots+Z_m\geq \rn+mE[Z])&=P(Z_1+\cdots+Z_m-mE[Z]\geq \rn)\\
&\leq \frac{1}{\rn^2}m\text{Var}[Z]\leq \frac{\nu^3}{\rn}.
\end{align*}
Since $\rn=\exp(cN^{d})$ for some $d>0$ we have $\sum_{N}\frac{\nu^3}{\rn}<\infty$ and by Borel-Cantelli Lemma, the left-hand side of (\ref{fenerik}) is bounded above by
\begin{equation}
 mE[Z]+\rn\leq m 2\nu+\rn\leq C\rn
\end{equation}
for all $N$ large enough for all $d\leq \nu$.
\end{proof}

\begin{proof}[Proof of Lemma \ref{aslanimik}]
We start with the proof of (\ref{cimbombir}). Note that for $i,j$ where $\lfloor i/\nu \rfloor =\lfloor j/\nu\rfloor$ we have
\begin{equation}\label{hadibak}
 \Lambda_{d}^0-\Lambda_{ij}^{1}=\left(1-\frac{2d}{N}\right)^p-\left(1-\frac{2p|i-j|}{N}\right)=\frac{2p(|i-j|-d)}{N}+O\left(\frac{d^2}{N^2}\right).
\end{equation}
The contribution from the second error term above is bounded by the right-hand side of (\ref{cimbombir}) by Lemma \ref{onceblok}. Hence, to finish the proof we need to control the contribution from the first error term. Define,
\begin{equation}
 \tilde{p}(d):=E_0[\sum_{i=0}^{\nu}(i-d)\mathbf{1}\{Q_i=d\}].
\end{equation}
Let us define $T_d^1=T_d$ and $T_d^{k}=\{i>T_d^{k-1}:Q_i=d\}$, for $k\geq 2$. Then we have
\begin{align}
 \nonumber E_0[\sum_{i=0}^{\nu}(i-d)\mathbf{1}\{Q_i=d\}]&=E_0[\sum_{k=1}^{\infty}(T_d^k-d)\mathbf{1}\{T^k_d<\nu\}]\\\nonumber&=E_0[\sum_{k=1}^{\infty}(T_d^k-T_d^1+T_d^1-d)\mathbf{1}\{T^k_d<\nu\}]\\&\label{esseksu}\leq E_0[(T_d-d)\mathbf{1}\{T_d< \nu \}]\left(1+E_d[\sum_{i=1}^{\infty}T_d^i\mathbf{1}\{T_d^i<\nu\}]\right).
\end{align}

It is easy to see that
$
 P_0(T_d=d)=1\cdot\frac{N-1}{N}\cdot\frac{N-2}{N}\cdots \frac{N-d+1}{N},
$
and thus,
$
 P_0(T_d)\leq Ce^{-d^2/N}.
$ Then $E_0[\sum_{k=1}^{\infty}(T_d^k-d)\mathbf{1}\{T^k_d<\nu\}]$ is bounded below by
$$
 E_0[\sum_{k=1}^{\infty}(T_d^k-d)\mathbf{1}\{T^k_d<\nu\}|T_d\not=d]P_0(T_d\not= d)\geq 2(1-P_0(T_d=d))\geq 2(1-Ce^{-d^2/N})\geq C\frac{d^2}{N}.
$$
Note that if $T_d\geq d+2k$ for some positive $k$ then the random walk $Q_i$ must make at least $k$ steps left. Since the probability of any step left is bounded by $d/N$ before reaching $d$, we have
\begin{equation*}
 P(T_d\geq d+2k)\leq {{d+2k}\choose{k}}\left(\frac{d}{N}\right)^k\leq C \frac{d^{2k}}{N^k}.
\end{equation*}
As a result we get
$$ E_0[(T_d-d)\mathbf{1}\{T_d< \nu \}]=\sum_{k=1}^{\infty}P_0(T_d\geq d+2k)\leq C\sum_{k=1}^{\infty}\frac{d^{2k}}{N^{k}}\leq C\left(\frac{1}{1-\frac{d^2}{N}}-1\right)\leq C \frac{d^2}{N}.
$$
Hence, we have
\begin{equation}
 C_1 \frac{d^2}{N}\leq E_0[(T_d-d)\mathbf{1}\{T_d< \nu \}]\leq C_2 \frac{d^2}{N}.
\end{equation}
Note that for the second term in (\ref{esseksu}) we have
\begin{align*} 1+E_d[\sum_{i=1}^{\infty}T_d^i\mathbf{1}\{T_d^i<\nu\}]&\leq 1+E_{d}[T_d\mathbf{1}\{T_d<\nu\}](1+E_d[\sum_{i=1}^{\infty}T_d^i\mathbf{1}\{T_d^i<\nu\}])\\&\leq \sum_{k=0}^{\infty}\left\{E_d[T_d\mathbf{1}\{T_d<\nu\}]\right\}^k.
\end{align*}
Also note that
$
 P_d(T_d=2k)\leq {{2k}\choose{k}}\left(\frac{\nu}{N}\right)^k
$. Using the bound ${{2k}\choose{k}}\leq Ck^{-1/2}2^{k}$ we get
\begin{equation*}
 E_d[T_d\mathbf{1}\{T_d<\nu\}]\leq C\sum_{k=1}^{\nu/2}2 k^{1/2}2^{k}\left(\frac{\nu}{N}\right)^k \leq C\sum_{k=1}^{\infty}\left(\frac{5\nu}{N}\right)^k\leq C\frac{\nu}{N}.
\end{equation*}
Hence, we have
\begin{equation}
 C_1 \frac{d^2}{N}\leq \tilde{p}(d)\leq C_2\frac{d^2}{N}.
\end{equation}
Now let us define the one-block contribution from the first error term in (\ref{hadibak})
\begin{equation}
 \sum_{i,j=1}^{\nu}(|i-j|-d)\mathbf{1}\{D_{ij}=d\}=:{\nu}^3 \tilde{Z}.
\end{equation}
Note that $\tilde{Z}\in[0,1]$. Hence, the contribution from the first error term to the left-hand side of (\ref{fenerik}) is stochastically bounded above by
$
 \frac{2p}{N}{\nu}^3\sum_{k=1}^{m}\tilde{Z}_k,
$
where $m=\lceil t\rn/\nu\rceil$ and $\tilde{Z}_k$ is a sequence of i.i.d. copies of $\tilde{Z}$. By above estimates we have
\begin{equation}
 C_1 \frac{d^2}{N}{\nu}^{-3}\leq E[\tilde{Z}]\leq C_2 \frac{d^2}{N}{\nu}^{-2}.
\end{equation}
Hence, using Hoeffding's inequality we get
\begin{equation*}
 P(\sum_{k=1}^{m}\tilde{Z}_k\geq 2mE[\tilde{Z}])\leq \exp(-2mE[\tilde{Z}])\leq \exp(-2m C_2 \frac{d^2}{N\nu^3}),
\end{equation*}
and by Borel-Cantelli Lemma we can conclude that the contribution from the first error term is a.s. bounded above by
\begin{equation}
  \frac{2p}{N}{\nu}^3 2mE[\tilde{Z}]\leq C\frac{2p}{N}{\nu}^3 \frac{\rn}{\nu}\frac{d^2}{N}\frac{1}{\nu^2}=C{\rur}\frac{d^2}{N^2},
\end{equation}
for all $N$ large enough. This finishes the proof of inequality (\ref{cimbombir}).

Next we prove the first part of Lemma \ref{aslanimik} that is inequality (\ref{cimbomik}). For ease of notation let us define
$ R:=t\rn,
$ and let us denote by $A_{d,\eta}(N)$ the term inside the curly bracket on the right-hand side of (\ref{cimbomik}), that is
\begin{equation}
 A_{d,\eta}(N):=t^2{\rur^2}2^{-N}\comnd+t\frac{\rur e^{\eta\antk ||d||}}{\nu\antk}.
\end{equation}
We can consider the couples $(i,j)$ with $i<j$ only. We first estimate the sum over pairs $(i,j)$ such that $j-i\geq K_N$. Since $j-i\geq\kur$ we have $\lfloor i/\nu\rfloor\not=\lfloor j/\nu\rfloor$. Thus, the left-hand side of (\ref{cimbomik}) is equal to (up to a constant)
\begin{equation}\label{ugras}
 \sum_{j-i\geq \kur}^{R}\mathbf{1}\{D_{ij}=d\}.
\end{equation}
Using Theorem \ref{mathew}, we have for any $d$ and $\eta$
\begin{equation*}
 E[\sum_{j-i\geq \kur}^{R}\mathbf{1}\{D_{ij}=d\}]=\sum_{j-i\geq\kur}^{R}p_{j-i}(d)\leq R^2 (2^{-N}\comnd+2^{-4N})\leq C_1 A_{d,\eta}(N).
\end{equation*}
Next, we estimate the variance of the sum (\ref{ugras})
\begin{align}
 \label{denet}\text{Var}[\sum_{j-i\geq \kur}^{R}&\mathbf{1}\{D_{ij}=d\}]=\sum_{\begin{array}{l}j_1-i_1\geq \kur\\j_2-i_2\geq \kur\end{array}}^{R}P(D_{i_1j_1}=D_{i_2 j_2}=d)-P(D_{i_1j_1}=d)P(D_{i_2 j_2}=d).
\end{align}

We can suppose that $i_1\leq i_2$. Note that if $i_1<j_1\leq i_2<j_2$ the right-hand side of (\ref{denet}) is zero. Hence, the only non-zero cases are when $i_1\leq i_2\leq j_1\leq j_2$ or $i_1\leq i_2\leq j_2\leq j_1 $. Let us consider the first case only since the second case can be done similarly. If $i_2-i_1\geq \kur$ or $j_2-j_1\geq \kur$, by Theorem \ref{mathew} the difference of probabilities in (\ref{denet}) is less than $2^{-4N}$. Hence, the sum in (\ref{denet}) over such couples is bounded by $R^2 2^{-4N}$ which is less than $N^{-3} A_{d,\eta}^2(N)$ for any $d$ and $\eta$.

Now, if $i_2-i_1<\kur$ and $j_2-j_1<\kur$ then by Theorem \ref{mathew}
\begin{equation}
 P(D_{i_1j_1}=D_{i_2 j_2}=d)\leq C2^{-N}\comnd.
\end{equation}
\newcommand{\ep}{\epsilon}

Now we investigate two separate cases. The first case is $||d||\leq (1-\ep\an)N/2$. For such $d$ using (\ref{berrat}) of the previous section and the reasoning in the proof of Lemma \ref{guzelyaz} we can conclude that for $N$ large enough
\begin{equation*}
 2^{-N}\comnd\leq C\exp\left\{-I\left((1-\an\ep)\frac{N}{2}\right)\right\}\leq C \exp\left(-c\ep^2\antk N\right),
\end{equation*}
for some $c$ independent of $\ep$. Thus, for any $\eta>0$ the right hand side of (\ref{denet}) is bounded above by $K_N^2 R^2 \exp\left(-c\ep^2\antk N\right)$ which asymptotically smaller than $N^{-3}A_{d,\eta}^{2}(N)$.

For the second case; $||d||\leq (1-\ep\an)N/2$, note that we have $|d-N/2|\leq \an\ep/2$. For such $d$
\begin{align*}
 2^{-N}\comnd&\geq C\exp\left\{-I\left((1-\an\ep)\frac{N}{2}\right)\right\}{N^{-1/2}}\geq C\exp(-c\antk\ep^2 N)N^{-1/2}.
\end{align*}
Hence, since
$
 \exp(c\antk\ep^2 N)R^2=C\an^{-2}N\exp(\antk(\beta^{-2}-c\ep^2) N),
$
 for $\ep$ small enough we have
$
 2^{-N}\comnd \gg N^8 R^{-2}.
$ As a result, the right hand side of (\ref{denet}) is bounded above by $C K_N^2 R^2 2^{-N}\comnd$ and we have
\begin{equation*}
C K_N^2 R^2 2^{-N}\comnd\ll N^{-3}R^4 2^{-2N}\comnd^2\leq N^{-3}A_{d,\eta}^2(N).
\end{equation*}
Hence, we have showed that
\newcommand{\ande}{A_{d,\eta}}
\begin{equation*}
 E[ \sum_{j-i\geq \kur}^{R}\mathbf{1}\{D_{ij}=d\}]\leq C_1 A_{d,\eta}(N),\hspace{0.1in}\text{Var}[\sum_{j-i\geq \kur}^{R}\mathbf{1}\{D_{ij}=d\}]\leq N^{-3}A_{d,\eta}^2(N),
\end{equation*}
for $N$ large enough. Thus, by Borel-Cantelli Lemma, for any $d\in\{0,\dots,N\}$ and $\eta>0$, sum over couples $(i,j)$ with $j-i\geq \kur$ is $\mathcal{Y}$ a.s. less than the right-hand side of (\ref{cimbomik}).

Now we consider the pairs $i,j$ where $j-i\leq \kur$. We separate two cases. First case is $||d||>(\log N)^{1+\epsilon}/\antk$. Since there are at most $\kur R$ couples with $j-i<\kur$ and $\kur R\ll R\nu^{-1}\an^{-2}e^{\antk\eta||d||}$ for $\forall \eta>0$ the inequality in (\ref{cimbomik}) holds true for those couples for such $d$.

For $||d||\leq (\log N)^{1+\epsilon}/\antk $ define $\bar{K}_N$ as $\bar{K}_N=\nu\left\lceil\frac{KN^{2}\log N}{\nu}\right\rceil.$
Then we have
\begin{equation*}
 KN^2\log N\leq \bar{K}_N N^2\log N < KNN^2\log N +\nu,
\end{equation*}
and thusly,
\begin{equation*}
 K\leq \bar{K}_N < K+\frac{\nu}{N^2 \log N}.
\end{equation*}
Since $\frac{\nu}{N^2 \log N}\overset{N\to\infty}\longrightarrow 0$ we have $K_N-K\overset{N\to\infty}\longrightarrow 0$. Hence, the difference is negligible and we will still use $K_N$ for $\bar{K}_N N^2 \log N$. Note that this way $K_N$ is a multiple of $\nu$.

For summation on (\ref{cimbomik}) over the pairs $j-i<K_N$ we have
\begin{equation}\label{kartal}
 \sum_{\underset{j-i<K_N}{\lfloor i/\nu \rfloor\not=\lfloor j/\nu\rfloor}}^{R}\mathbf{1}\{\ydij=d\}\leq\sum_{k=0}^{\lceil K_N\rceil}\sum_{l=0}^{\lceil R/K_N \rceil}\sum_{m=j_k}^{\lceil \kur\rceil}\mathbf{1}\{D_{lK_N+k,lK_N+k+m}=d\},
\end{equation}
where $j_k$ is the smallest integer such that
$
 \left\lfloor \frac{lK_N+k}{\nu}\right\rfloor\not= \left\lfloor\frac{lK_N+k+j_k}{\nu}\right\rfloor,
$
which does not depend on $l$. Define the random variables $Z_l(j,d)$ as
\begin{equation*}
 Z_l(j,d)=\frac{1}{\lceil\kur\rceil}\sum_{m=j}^{\lceil \kur\rceil}\mathbf{1}\{D_{lK_N+k,lK_N+k+m}=d\}.
\end{equation*}
Note that $(Z_l(j,d),\;l\in\mathbb{N})$ is an i.i.d. sequence of random variables in $[0,1]$ for fixed $j,k$ and $d$.
\newcommand{\dd}{||d||}
\newcommand{\lepn}{(\log N)^{1+\ep}/\antk}

Let $E_N:=\{d:\dd < \lepn ,\; d \geq N/2\}$. Denote by $\mathbf{1}$ the vertex on the hypercube with all coordinates is equal to 1. Define $T_{\bf 1}=\min\{k\geq 1:\;\yn(k)={\bf 1}\}$. Let $z_d$ be any vertex of the hypercube with $\dist(z_d,\mathbf{1})=d$. For $d\in E_N$
\begin{equation}\label{cisse}
 \mathbb{P}[Z_l(j_k,d)>0]\leq \comnd \mathbb{P}_{z_d}(T_{\mathbf{1}}< K_N) \leq \comnd e^{\lambda K \log(N)}\mathbb{E}_{z_d}[e^{-\lambda T_{\mathbf{1}}/N^{2}}].
\end{equation}
According to Lemma 3.4 of \cite{cerg},
\begin{equation}
 \mathbb{E}_{z_d}[\exp(-\lambda T_{\mathbf{1}}/m(N))]\leq (2^{-N}m(N)\lambda^{-1}+\xi_N(d))(1+o(1))
\end{equation}
for $N\log(N)\ll m(N)\ll 2^N$, with $\xi_N(d)=2^{-N}\frac{N}{2}\comnd^{-1}\sum_{j=1}^{N-d}{{N}\choose{d+j}}\frac{1}{j}$.
Here we take $m(N)=N^2$. Since for $d\in E_N$, $N/2\leq d$ and $N-d\leq \lepn$ we have
\begin{align*}
 \xi_N(d)\leq 2^{-N}\frac{N}{2}\comnd^{-1} \sum_{j=1}^{N-d}{{N}\choose{d+j}}\leq 2^{-N}\frac{N}{2}\comnd^{-1}(N-d)\comnd\leq C 2^{-N} N \frac{(\log N)^{1+\epsilon}}{\antk}.
\end{align*}
Hence, for those $d$, for any $\epsilon'>0$ small and $N$ large enough
$
 N^{\lambda K}\mathbb{E}_d[e^{-\lambda T_{\mathbf{1}}/N^{2}}]\leq 2^{-N(1-\epsilon')},
$ and by (\ref{cisse})
\begin{equation}
 \mathbb{P}[Z_{l,k}(j_k,d)>0]\leq {{N}\choose{\lceil\lepn\rceil}} 2^{-N(1-\epsilon')}.
\end{equation}
Hence, the probability of the right-hand of (\ref{kartal}) is bounded above by
\begin{equation}
 \label{anindaiki}\mathbb{P}\left(\sum_{k=0}^{\lceil K_N\rceil}\sum_{l=0}^{\lceil R/K_N \rceil}Z_l(k,d)>0\right)\leq
R {{N}\choose{\lceil\lepn\rceil}} 2^{-N(1-\epsilon')}.
\end{equation}
Note that since $\lepn\ll N$ we have
\begin{equation}
 2^{-N}{{N}\choose{\lceil\lepn\rceil}}\ll \exp(-cN),
\end{equation}
for some constant $c>0$. Also, $R= C N^{1/2} \an^{-1}\exp(\an^2\beta^{-2} N/2)$. As a result we have
\begin{equation}
 (\ref{anindaiki})\leq 2^{-\epsilon'' N},
\end{equation}
for some $\epsilon''>0$. Hence, by Borel-Cantelli those $d$ are not even found by the random walk and satisfy inequality (\ref{cimbomik}) for any $\eta>0$.

Now for the case $d\leq \lepn$ we look at two separate cases: $j_k\leq 2d$ and $j_k>2d$. For the first case note that the number of $k$'s in $\{1,\dots,\kur\}$ s.t. $j_k\leq 2d$ is at most $\kur d/\nu$ . Also note that then $Z_{l}(j_k,d)\leq Z_l(0,d)$. Using the fact
\begin{equation*}
 \frac{1}{NK_N}\leq E[\sum_{i=1}^{K_N}\mathbf{1}\{D_{1i}=d\}]\leq C\frac{1}{K_N},
\end{equation*}
we get
\begin{equation*}
 C_1 \frac{1}{N\kur}\leq E[Z_l(0,d)]\leq C_2 \frac{1}{\kur}.
\end{equation*}
Thus, by Hoeffding's inequality
\begin{align*}
 P(\sum_{l=0}^{\lceil \rn/{K_N}\rceil}Z_l(0,d)\geq 2 \frac{\rn}{K_N}E[Z_l(0,d)])&\leq \exp(-2\rn \nu^{-1}E[Z_l(0,d)])\\&\leq \exp(-C \rn \nu^{-1} \frac{1}{K_N}),
\end{align*}
which decreases at least exponentially. Hence, by Borel-Cantelli Lemma we have
\begin{equation*}
 \kur \sum_{l=0}^{\lceil \rn/{K_N}\rceil}Z_l(0,d) \leq 2 \kur \frac{\rn}{K_N}E[Z_l(0,d)]) \leq C \frac{\rur}{\kur},
\end{equation*}
for $k$ s.t. $j_k\leq 2d$.

Now we consider $k$ s.t. $j_k\geq 2d$. Note that for $j\geq 2d$ we have $Z_{l}(j,d)\leq Z_{l}(d+6,d)$ for $N$ large enough and
\begin{equation}
 C_1 N^{-6}\leq \kur E[Z_{l}(d+6,d)]\leq C_2 N^{-3}.
\end{equation}
Hence, by Hoeffding's inequality, for $k$ s.t. $j_k\geq 2d$
\begin{equation}
 P(\sum_{l=0}^{\lceil\rn/{K_N}\rceil}Z_{l}(j_k,d)\geq \frac{\rn}{N^3 K_N})\leq \exp(-C \frac{\rn}{K_N}N^{-6}),
\end{equation}
which decreases at least exponentially with $N$. Hence, by Borel-Cantelli Lemma a.s.
\begin{equation}
 \kur \sum_{l=0}^{\lceil \rn/{K_N}\rceil}Z_{l}(j_k,d) \leq C\frac{\kur \rn}{N^{3}K_N}\leq C \frac{\rur}{N^3},
\end{equation}
for all $N$ large enough. Hence, summing over $k$ we get
\begin{align}\label{essekk}
 \sum_{k=0}^{\lceil \kur \rceil}\kur \sum_{l=0}^{\lceil \rn/{K_N}\rceil}Z_l(0,d)&\leq \frac{\kur d }{\nu}\frac{\rur}{\kur}+C\frac{\rur}{N^3}
\leq C \left\{\frac{d\rn}{\nu} + \frac{\rur}{N^3}\right\}.
\end{align}
Now, since $N^{3}\gg \nu\antk$ for any $\eta>0$ there exists a constant $C$ s.t. for $N$ large enough
\begin{equation*}
 \frac{\rur}{N^3}\leq C \frac{\rn e^{\eta\antk||d||}}{\nu\antk}.
\end{equation*}
Recall that since $d\leq \lepn$ we have $||d||=d$. For any $\eta>0$ we can find a constant $C$ s.t.
\begin{equation*}
 x\leq Ce^{\eta x}, \;\forall x\geq 0.
\end{equation*}
Using this fact with $x=d\antk$ we can conclude that for any $\eta>0$ given there exists a constant $C$ s.t.
\begin{equation*}
 \frac{d\rn}{\nu}\leq C \frac{\rn e^{\eta\antk||d||}}{\nu\antk}.
\end{equation*}
Hence, for any $\eta>0$ (\ref{essekk}) is bounded above by the right-hand sight of (\ref{cimbomik}) for all large enough $N$ with a large enough constant $C$. This finishes the proof of inequality (\ref{cimbomik}) and hence, the proof of Lemma \ref{aslanimik}.
\end{proof}

\section{Proofs of Theorem 1 and Theorem 2}\label{sectionfive}

In this section we prove Theorems \ref{theop22} and \ref{theop1}. We will first prove Theorem \ref{theop1}, that is, we will prove that, under the non-linear normalization of Theorem \ref{theop1}, the maximal and the clock processes converge to the same extremal process on the space $D([0,T],\R)$ quipped with the $M_1$ topology. Therefore, we start this section by recalling the definitions and basic properties of the extremal processes and the $M_1$ topology.

\subsection{Extremal processes}
Consider a probability distribution function $F(x)$. Define a family of finite dimensional distributions $F_{t_1,\dots,t_k}(x_1,\dots,x_k)$ for $k\geq 1$, $0<t_1\leq\cdots\leq t_l$ and $x_i\in\mathbb{R}$ by
\qw\label{family}
F_{t_1,\dots,t_l}(x_1,\dots,x_l)=F^{t_1}(\wedge_{i=1}^l x_i)\;F^{t_2-t_1}(\wedge_{i=2}^k x_i)\cdots F^{t_l-t_{l-1}}(x_l),
\qwe
where $\wedge$ stands for minimum. The family (\ref{family}) forms a consistent family of finite dimensional distributions. Hence, by Kolmogorov's extension theorem there exists a continuous time stochastic process $(Y(t),\;t>0)$ with finite dimensional distributions given by (\ref{family}). $(Y(t),\;t>0)$ is called the extremal process generated by $F$ or $F$-extremal.

We will consider the probability distribution $G_{\beta}(x)$ given by
\qw\label{gdist}
G_{\beta}(x)=\left\{\begin{array}{ll}\exp(-1/x^{1/\beta^2})&x>0,\\0&x\leq 0.\end{array}\right.
\qwe
Since the support of $G_{\beta}$ is non-negative numbers, we can extend the extremal process $(Y_{\beta}(t),t>0)$ generated by $G_{\beta}$ to $(Y_{\beta}(t),t\geq0)$ by defining $Y_{\beta}(0)=0$ for all realizations. Thus, by $(ii)$ and $(iii)$ of Proposition 4.7 on page 180 of \cite{res}, $(Y_{\beta}(t),\;t\geq 0)$ has a version in $D([0,\infty),[0,\infty))$, the space of non-negative c\`{a}dl\`{a}g functions on $[0,\infty)$. For the rest of the paper we will call $(Y_{\beta}(t),t\geq0)$ the extremal process generated by $G_{\beta}$ or $G_{\beta}$-extremal where $G_{\beta}$ is given by (\ref{gdist}).

Note that in order to check that a stochastic process $(Y(t),t\geq 0)$ has the finite dimensional distributions of the $G_{\beta}$-extremal process it is enough to check that a.s. $Y(0)=0$ and $Y(t)$ is non-decreasing,
 and for any $l\geq 1$, $0=t_0< t_1\leq\cdots\leq t_l$ and $0<x_1\leq \cdots \leq x_l$
\qw\label{extrdist}
\mathbb{P}(Y(t_1)\leq x_1,\cdots,Y(t_l)\leq x_l)=\prod_{k=1}^{l}\exp\left(-\frac{t_k-t_{k-1}}{x_k^{1/\beta^2}}\right).
\qwe

\subsection{$J_1$ and $M_1$ topologies}

Let $D=D([0,T],\mathbb{R})$; the space of c\`{a}dl\`{a}g functions. The usual Skorohord $J_1$ topology is given by the metric $d_{J_1}$ where
\qw
d_{J_1}(f,g)=\inf_{\lambda\in\Lambda}\{||\lambda-I||_{\infty}\vee||f\circ\lambda-g||_{\infty}\}.
\qwe
Here $\Lambda$ is the set of strictly increasing functions from $[0,T]$ onto $[0,T]$ that are continuous with a continuous inverse, and $I$ is the identity map on $[0,T]$.

The $M_1$ topology is also given by a metric, $d_{M_1}$. For $f\in D$ we define its completed graph $\Gamma_f$ by
\qw
\Gamma_f:=\{(t,z)\in [0,T]\times [0,\infty):\;z=\alpha f(t-)+(1-\alpha)f(t) \text{ for some }\alpha\in[0,1]\}.
\qwe
We can order points of $\Gamma_f$ as follows: $(t_1,z_1)\leq (t_2,z_2)$ if either i) $t_1<t_2$ or ii) $t_1=t_2=t$ and $|f(t-)-z_1|\leq |f(t-)-z_2|$. Let $\Pi_f$ be the set of nondecreasing continuous functions $(r,u)$ from $[0,1]$ onto $\Gamma_f$, with $r$ being the time component and $u$ being the spatial component. Here $(r,u)$ is nondecreasing for the order on $\Gamma_f$ we have just defined. Than the metric $d_{M_1}$ is given as follows:
\qw
d_{M_1}(f_1,f_2)=\inf\{||u_1-u_2||_\infty\vee ||r_1-r_2||_\infty:\;(r_1,u_1)\in\Pi_{f_1},(r_1,u_2)\in\Pi_{f_2}\}.
\qwe
It is easy to see that $d_{M_1}(f_1,f_2)\leq d_{J_1}(f_1,f_2)$ for all $f_1,f_2\in D$. On the other hand, $M_1$ topology is weaker than the $J_1$ topology. As an example consider the sequence of functions
\qw
f_n={\bf 1}\{[1-1/n,1)\}+2{\bf 1}\{[1,T]\}.
\qwe
$f_n$ converges to $f=2{\bf 1}\{[1,T]\}$ in $M_1$ topology but does not convergence in $J_1$ topology.

For tightness characterizations we need the following definitions:
\begin{align}
&w_f(\delta)=\sup\{\min(|f(t_1)-f(t)|,|f(t)-f(t_2)|):t_1\leq t \leq t_2\leq T, t_2-t_1\leq\delta\},\\
&w'_f(\delta)=\sup\{\inf_{\alpha \in [0,1]}|f(t)-(\alpha f(t_1)+(1-\alpha)f(t_2))|:t_1\leq t \leq t_2\leq T, t_2-t_1\leq \delta\},\\
& v_f(t,\delta)=\sup \{|f(t_1)-f(t_2)|:t_1,t_2\in [0,T]\cap (t-\delta,t+\delta)\}.
\end{align}
The following is from Theorem 12.12.3 of \cite{whi} and Theorem 15.3 of \cite{bil}.
\begin{theorem}\label{gottokiki}
The sequence of probability measures $\{P_n\}$ on $D([0,T],\mathbb{R})$ is tight in the $J_1$-topology if

(i) For each positive $\epsilon$ there exists $c$ such that
\begin{equation}
 P_n[f:||f||_\infty> c]\leq \epsilon, \hspace{0.4in} n\geq 1
\end{equation}

(ii) For each $\epsilon>0$ and $\eta >0$, there exists a $\delta$, $0<\delta<T$, and an integer $n_0$ such that
\begin{equation}\label{budane}
 P_n[f: w_f(\delta)\geq \eta]\leq \epsilon, \hspace{0.4in} n\geq n_0
\end{equation}
and
\begin{equation}
 P_n[f:v_f(0,\delta)\geq \eta]\leq \epsilon \; \text{and} \; P_n[f:v_f(T,\delta)\geq \eta]\leq \epsilon, \;\; n\geq n_0
\end{equation}
The same claim holds for the $M_1$ topology with $w_f(\delta)$ in (\ref{budane}) is replaced by $w_f'(\delta)$.
\end{theorem}

\subsection{Proof of Theorem \ref{theop1}} We will first prove the convergence of the maximal process and then prove that the clock process is dominated by the maximal process. For the former, we start with proving the convergence of the finite
dimensional distributions to (\ref{extrdist}) of the extremal process $Y_\beta$, using the comparison results of Section 3.

Let us define
\begin{equation}
t_k(N)=\lfloor t_k\rn\rfloor -1,\;\;k=1,\dots,l.
\end{equation}

\begin{proposition}\label{cukkafa}
For every sequence $\{t_k\}$ and $\{x_k\}$ i.e. $0=t_0< t_1\leq t_2\leq \dots \leq t_l=T$ and $0< x_1 \leq x_2\leq \dots\leq x_l$, under the assumptions of Theorem \ref{theop1}, $\mathcal{Y}$ a.s.
\begin{align}
\label{hausdorff} \mathbb{P}(\max_{i\leq t_1(N)} {\xni}\leq C_N(x_1),\dots,\max_{i\leq t_l(N)}{\xni}\leq C_N(x_l)|\mathcal{Y})\overset{N\to\infty}\longrightarrow \prod_{k=1}^{l}\exp\left(-\frac{(t_{k}-t_{k-1})K}{x_k^{1/\beta^2}}\right),
\end{align}
where $K=2\beta^{-2} p$ is as in Proposition \ref{esasli}.
\end{proposition}
\begin{proof}
We use the results of Section 2 to prove the convergence of extremal statistics of $\xnni$ both on the whole $i\in\mathbb{N}$ and on the resampled cloud $w_{\rho}$. 

Let $J_N(k):=\lfloor t_k\rn/\nu\rfloor$, $k=0,\dots,l$. Then clearly the left-hand side of (\ref{hausdorff}) is bounded above by
\begin{equation}\label{esekami}
 \pro(\max\{X_N^1(j\nu+i):j=J_N(k-1),\dots,J_N(k)-1,\;i=0,\dots,\nu\}\leq C_N(x_k):k=1,\dots,l)
\end{equation}
By block independence of $X_N^1$ and Proposition \ref{esasli}, for any $\delta>0$, for $N$ large enough 
\begin{align*}
(\ref{esekami})&= \prod_{k=1}^l \pro(\max_{i=1,\dots,\nu}X_N^1(i)\leq C_N(x_k))^{J_N(k)-J_N(k-1)}\\&
\leq \prod_{k=1}^l\left(1-\frac{(1-\delta)\nu}{\rn}\frac{K}{x_k^{1/\beta^2}}\right)^{(t_k-t_{k-1})\rn/\nu}\\&
\leq \prod_{k=1}^l\exp\left(-\frac{(t_k-t_{k-1})K(1-2\delta)}{x^{1/\beta^2}}\right).
\end{align*}
A lower bound can be achieved similarly. Hence, 
\begin{equation}
\label{mer1}\pro(\max_{i\leq t_1(N)} {\xnni}\leq C_N(x_1),\dots,\max_{i\leq t_l(N) }{\xnni}\leq C_N(x_l))\overset{N\to\infty}\longrightarrow \prod_{k=1}^{l}\exp\left(-\frac{(t_{k}-t_{k-1})K}{x_k^{1/\beta^2}}\right)
\end{equation}
Similarly, using Lemma 2 we have for $N$ large enough
\begin{equation}
\label{merm2}
\pro(\max_{\underset{i\in w_{\rho}}{i\leq t_1(N)}} {\xnni}\leq C_N(x_1),\dots,\max_{\underset{i\in w_{\rho}}{i\leq t_l(N)} }{\xnni}\leq C_N(x_l))\leq \prod_{k=1}^{l}\exp\left(-\frac{(t_{k}-t_{k-1})KC(\rho)}{x_k^{1/\beta^2}}\right).
\end{equation}

Note that $\max_{{i\leq t \rn, i\in w_{\rho}}} {\xni}$ is bounded above by $\max_{{i\leq t \rn}} {\xni}$. Hence, by (\ref{merm2}) and part $(i)$ of Proposition \ref{berker}, for any $\epsilon>0$ given we have $\mathcal{Y}$ a.s.
\qw
\mathbb{P}(\max_{i\leq t_1(N)} {\xni}\leq C_N(x_1),\dots,\max_{i\leq t_l(N)}{\xni}\leq C_N(x_l)|\mathcal{Y})\leq e^{-\frac{KC(\rho)t_1}{x_1}}\cdots e^{-\frac{KC(\rho)(t_l-t_{l-1})}{x_l}}+\epsilon
\qwe
for all $N$ large enough. On the other hand, by (\ref{mer1}) and part $(ii)$ of Proposition \ref{berker}, for $\epsilon>0$ given we have $\mathcal{Y}$ a.s.
\qw
e^{-\frac{Kt_1}{x_1}}\cdots e^{-\frac{K(t_l-t_{l-1})}{x_l}}-\epsilon\leq \mathbb{P}(\max_{i\leq t_1(N) } {\xni}\leq C_N(x_1),\dots,\max_{i\leq t_l(N)}{\xni}\leq C_N(x_l)|\mathcal{Y})
\qwe
for all $N$ large enough. Recall that by Lemma \ref{onder} we have $\lim_{\rho\to\infty}C(\rho)=1$. Hence, letting $\rho\to\infty$ and $\epsilon\to0$ finishes the proof of Proposition \ref{cukkafa}.
\end{proof}

\begin{proof}[Proof of Theorem \ref{theop1} part (i)]

Let $\rsn(\cdot)$ and $\bar{m}_N(\cdot)$ denote the rescaled clock process and maximal process, respectively, that is
\begin{equation}
 \rsn(\cdot)=\frac{S_N(\lfloor \rn\;\cdot\rfloor)}{\tn},\hspace{0.3in}\bar{m}_N(\cdot)=\frac{m_N(\lfloor \rn\;\cdot\rfloor)}{\tn}.
\end{equation}
Recalling the definition of $m_N(k)$ and the definition of $C_N(x)$ it is easy to see that
\qw
(\bar{m}_N(t))^{\an}\leq x \Longleftrightarrow \max_{i\leq \lfloor t\rn\rfloor -1}X_N^0(i)\leq C_N(x).
\qwe
Also, by definition $\bar{m}_N(0)=0$. Hence, since $(\bar{m}_N(\cdot))^{\an}$ is non-decreasing, we get the convergence of the finite dimensional distributions by Proposition \ref{cukkafa}.
\newcommand{\ep}{\epsilon}

We use the Theorem \ref{gottokiki} to check tightness. Since the process $(\bar{m}_N(\cdot))^{\an}$ is non-decreasing to check condition (i) it is enough to check that $(\bar{m}_N(T))^{\an}$ is tight. In this case, the convergence of fixed time distribution gives the desired tightness result.

Since $\bar{m}_N^{\an}$ is increasing $w_{f}(\delta)$ is 0. As a result in order to check (ii) we have to control $v_{\bar{m}_N^{\an}}(0,\delta)$ and $v_{\bar{m}_N^{\an}}(T,\delta)$. Again because of the monotonicity, controlling $v_{\bar{m}_N^{\an}}(0,\delta)$ boils down to check that $\mathbb{P}[(\bar{m}_N(\delta))^{\an}\geq \eta]\leq \epsilon$ for small enough $\delta$ and large enough $N$. However, by convergence of the fixed time distribution it is same as  checking $\mathbb{P}[Y(K\delta)\geq \eta]\leq \epsilon/2$. We have
\begin{equation}
 \mathbb{P}[Y(K\delta)\geq \eta]=1-\exp(-\frac{K\delta}{\eta^{1/\beta^2}}).
\end{equation}
Clearly for small enough $\delta$ the probability above is less than $\epsilon/2$.

Similarly controlling $v_{\bar{m}_N^{\an}}(T,\delta)$ boils down to find $\delta$ small enough so that
\begin{equation}\label{fosfos}
 \mathbb{P}[Y(K T)-Y(K (T-\delta))\geq \eta]\leq \epsilon/2.
\end{equation}
Observe that
\begin{equation}
 \mathbb{P}[Y(KT)-Y((K(T-\delta)))=0]=\int_{0}^\infty e^{-\frac{K(T-\delta)}{x^{1/\beta^2}}}e^{-\frac{K\delta}{x^{1/\beta^2}}}\frac{T-\delta}{x^{1/\beta^2+1}}\text{d}x=\frac{T-\delta}{T},
\end{equation}
then
\begin{equation*}
 \mathbb{P}[Y(K T)-Y(K (T-\delta))\geq \eta]\leq 1-\mathbb{P}[Y(KT)-Y((K(T-\delta)))=0]=\frac{\delta}{T}.
\end{equation*}
Hence, (\ref{fosfos}) follows by taking $\delta\leq T\ep/4$.
\end{proof}
\begin{proof}[Proof of Theorem \ref{theop1} part (ii)]
We start the proof by showing that the clock process is dominated by the maximal in the following sense:
\begin{lemma}\label{godone}
For any $t_0,\delta>0$ given, $\mathcal{Y}$ a.s. there exists a constant $A(t_0,\delta,c)$ s.t. for $N$ large enough
\begin{equation}
\mathbb{P}(S_N(t\rn)\leq A_N m_N(t\rn)\;\forall t \in [t_0,T]|\mathcal{Y})\geq 1-\delta,
\end{equation}
where $A_N=\an^{-2}A$.
\end{lemma}
\begin{proof}
By Proposition \ref{cukkafa} we can choose $\epsilon>0$ small enough so that $\mathcal{Y}$ a.s.
\begin{equation}
 \pro[\frac{m_N(t_0\rn)}{\tn}\geq \epsilon^{1/\an}|\mathcal{Y}]\geq 1-\delta/4.
\end{equation}
for all $N$ large enough.
Let us denote by $B_N$ the event inside the probability above. We partition $S_N(t\rn)$ according to this $\epsilon$ as follows
\begin{align}
 \label{kazim} S_N(t\rn)&=\sum_{i=1}^{t\rn}e_i e^{\beta\sqrt{N}X_N^0(i)} \mathbf{1}\{e^{\beta\sqrt{N}X_N^0(i)}\leq \tn \epsilon^{1/\an}\}
\\&+\sum_{i=1}^{t\rn} e_i e^{\beta\sqrt{N}X_N^0(i)} \mathbf{1}\{e^{\beta\sqrt{N}X_N^0(i)}> \tn \epsilon^{1/\an}\}.
\end{align}
We have
\begin{align}
\label{essek} \mathbb{E}[\sum_{i=1}^{t\rn}\mathbf{1}\{e^{\beta\sqrt{N}X_N^0(i)}> \tn \epsilon^{1/\an}\}|\mathcal{Y}] & \leq  \sum_{i=1}^{T\rn}\pro[e^{\beta\sqrt{N}X_N^0(i)}>\tn\epsilon^{1/\an}|\mathcal{Y}]\leq   C T \an^{-2}.
\end{align}
Using (\ref{essek}) and Cheybshev inequality we get $\mathcal{Y}$ a.s.
\begin{align}
\nonumber\mathbb{P}[\sum_{i=1}^{t\rn}e_i\mathbf{1}\{e^{\beta\sqrt{N}X_N^0(i)}> \tn \epsilon^{1/\an}\}\geq \an^{-2}A_1|\mathcal{Y}]
 &\leq\mathbb{E}[\frac{\an^{2}}{A_1}\sum_{i=1}^{t\rn}\mathbf{1}\{e^{\beta\sqrt{N}X_N^0(i)}>\tn\epsilon^{1/\an}\}|
 \mathcal{Y}]\\\label{stan}&\leq \frac{CT}{A_1}.
\end{align}
Now we define the event
\begin{equation}\label{yarim}
 C_N:=\{\sum_{i=1}^{t\rn}e_i\mathbf{1}\{e^{\beta\sqrt{N}X_N^0(i)}> \tn \epsilon^{1/\an}\}\leq \an^{-2}A_1\}.
\end{equation}
Hence, using (\ref{stan}) if we choose $A_1$ large enough we have $\mathcal{Y}$ a.s.
$
 \mathbb{P}(C_N|\mathcal{Y})\geq 1-\delta/4,
$
for all $N$ large enough.
Then on $C_N$ we have
\begin{align}\label{uykubir}
 \sum_{i=1}^{t\rn} e_i e^{\beta\sqrt{N}X_N^0(i)} \mathbf{1}\{e^{\beta\sqrt{N}X_N^0(i)}> \tn \epsilon^{1/\an}\}\leq A_1 \an^{-2} m_N(t\rn).
\end{align}
Considering (\ref{kazim}) we have for $C$ large enough that does not depend on $t$ 
\newcommand{\kab}{\mathbf{1}}
\begin{align}
  \nonumber\mathbb{E}[\sum_{i=1}^{t\rn}e_i e^{\beta\sqrt{N} X_N^0(i)}\kab\{e^{\beta\sqrt{N}X_N^0(i)}\leq \tn \epsilon^{1/\an}\}|\mathcal{Y}]& =\sum_{i=1}^{t\rn}\mathbb{E}[e^{\beta\sqrt{N} Z} 
\kab\{e^{\beta\sqrt{N} Z}\leq \tn \epsilon^{1/\an}\}]\\ \label{kizak} &\leq C \tn t\an^{-2}\epsilon^{1/\an},
\end{align}
where $Z$ is a standard normal random variables. Let us define the sequence of events
\begin{equation}
 D_N:=\{\sum_{i=1}^{t\rn}e_i e^{\beta\sqrt{N} X_N^0(i)}1\{e^{\beta\sqrt{N}X_N^0(i)}\leq \tn \epsilon^{1/\an}\}\leq A_2 \tn  \an^{-2}\epsilon^{1/\an}\}.
\end{equation}
Using (\ref{kizak}) and Cheybshev inequality, we have for $A_2$ large enough $\mathcal{Y}$ a.s.
$
 \pro[D_N|\mathcal{Y}] \geq 1-\delta/4,
$ for all $N$ large enough. Note that on the intersection of $B_N$ and $D_N$ we have
\begin{equation}\label{uykuiki}
 \sum_{i=1}^{t\rn}e_i e^{\beta\sqrt{N} X_N^0(i)}1\{e^{\beta\sqrt{N}X_N^0(i)}\leq \tn \epsilon^{1/\an}\}\leq A_2  \an^{-2} m_N(t\rn),
\end{equation}
since on $B_N$ it is the case that $\tn\epsilon^{1/\an}\leq m_N(t\rn)$. Let $A=A_1+A_2$. Then (\ref{uykubir}) and (\ref{uykuiki}) finishes the proof of Lemma \ref{godone} since $\mathcal{Y}$ a.s.
$
 \mathbb{P}(B_N,C_N,D_N|\mathcal{Y})>1-\delta,
$ for all $N$ large enough.
\end{proof}

Lastly, we show that the rescaled processes non-linearly normalized by taking the $\an$th power, $(\rsn(\cdot))^{\an}$ and $(\bar{m}_N(\cdot))^{\an}$, are asymptotically close to each other in Skorohord $J_1$ distance.
\begin{lemma}\label{ozen}
 For $ \epsilon >0$ small enough $\mathcal{Y}$ a.s. for $N$ large enough
\begin{equation}
\pro[\sup_{t\in[0,T]}|\bar{S}_N(t)^{\an} - \bar{m}_N(t)^{\an}|\geq \epsilon]\leq \epsilon
\end{equation}
\end{lemma}
\begin{proof}
First note that since $\bar{S}_N(t)\geq \bar{m}_N(t)$ for all $t$ we have
\begin{equation}
|\bar{S}_N(t)^{\an} - \bar{m}_N(t)^{\an}|=\left(\bar{S}_N(t)^{\an} - \bar{m}_N(t)^{\an}\right).
\end{equation}
\newcommand{\ep}{\epsilon}
Let $t_0>0$. We partition the sum $S_N(t_0)$ as before:
\begin{align}\label{bakalim}
 S_N(t_0) &=\sum_{i=1}^{t_0\rn} e_i e^{\beta\sqrt{N} X_N^0(i)} 1\{e^{\beta\sqrt{N}X_N^0(i)}\leq \tn\ep^{1/\an}\}\\&\label{bakalimiki}
+\sum_{i=1}^{t_0\rn} e_i e^{\beta\sqrt{N} X_N^0(i)}1\{e^{\beta\sqrt{N}X_N^0(i)}> \tn\ep^{1/\an}\}.
\end{align}
As in the proof of the previous proposition we have
\begin{equation}
 \mathbb{E}[\sum_{i=1}^{t_0\rn}e_i e^{\beta\sqrt{N} X_N^0(i)} 1\{e^{\beta\sqrt{N}X_N^0(i)}\leq \tn\ep^{1/\an}\}]\leq Ct_0\tn\an^{-2}\ep^{1/\an},
\end{equation}
and as a consequence
\begin{equation}
 \pro[\sum_{i=1}^{t_0\rn}e_i e^{\beta\sqrt{N} X_N^0(i)} 1\{e^{\beta\sqrt{N}X_N^0(i)}\leq\tn \ep^{1/\an}\} > \an^{-2}\tn \ep^{1/\an}]< C t_0.
\end{equation}
Also we have as in the same proof
\begin{equation}
 \pro[\max_{i\leq t_0\rn}e^{\beta\sqrt{N}X_N^0(i)}>\tn\ep^{1/\an}]\leq 1-e^{-\frac{K t_0}{\epsilon^{1/\beta^2}}}\leq \ep.
\end{equation}
We choose $t_0$ small enough so that $Ct_0\leq \ep/2$. Then, on a set of probability less than $\ep$ we have that (\ref{bakalim}) is less than $\an^{-2}\tn \ep^{1/\an}$ and (\ref{bakalimiki}) is zero. Now if we choose $N$ large enough so that $\an^{2\an}$ close to 1 we have
\begin{equation}
 \pro[\sup_{t\in[0,t_0]}\left(\bar{S}_N(t)^{\an} - \bar{m}_N(t)^{\an}\right) >\epsilon]\leq \pro[\bar{S}_N(t_0)>\ep^{1/\an}]\leq \ep.
\end{equation}
For $t\in[t_0,T]$, using Lemma \ref{godone} there exists an $A$ such that
\begin{equation}
 \frac{S_N(t\rn)}{\tn}\leq A \an^{-2} \frac{m_N(t\rn)}{\tn} \;\; \forall t\in[t_0,T],
\end{equation}
on a set that has probability greater than $1-\ep$. On this event we have
\begin{align*}
 &\pro [\sup_{t\in(t_0,T]}(\bar{S}_N(t))^{\an}-(\bar{m}_N(t))^{\an}>\ep]\leq  \pro[\sup_{t\in(t_0,T]}\left(\left(\frac{A}{\antk}\right)^{\an}-1\right) \bar{m}_N(t)^{\an}>\ep].
\end{align*}
Note that $\left(\frac{A}{\antk}\right)^{\an}\overset{N\to\infty}\longrightarrow 1$. We choose $N$ large enough so that $\left(\left(\frac{A}{\antk}\right)^{\an}-1\right)\leq \ep^2/2$. 
Using the monotonicity of $\bar{m}_N$ and the convergence of $(\bar{m}_N)^{\an}$ to the extremal process $Y(K\cdot)$ we can conclude that up to a small error the last line above is less than
\begin{equation*}
 \pro[\ep^2Y_1(KT)>\epsilon]=1-e^{-{KT\ep^{1/\beta^2}}},
\end{equation*}
which is small for small enough $\ep$. This finishes the proof of Lemma \ref{ozen}.
\end{proof}
Now we can finish the proof Theorem \ref{theop1}. By Lemma \ref{ozen} we have
\begin{equation}
 d_{J_1}((\rsn(\cdot))^{\an},(\bar{m}_N(\cdot))^{\an})\overset{\text{p}}\longrightarrow 0\hspace{0.2in}\text{as }N\to\infty,
\end{equation}
where $\overset{\text{p}}\to$ stands for convergence in probability. This convergence in probability of the Skorohord $J_1$ distance between $\bar{S}_N^{\an}$ and $\bar{m}_N^{\an}$, and 
the convergence of $\bar{m}_N(\cdot)^{\an}$ to $Y(K\cdot)$ on $M_1$ topology finishes the proof of part (ii) of Theorem \ref{theop1}.
\end{proof}

\subsection{Proof of Theorem \ref{theop22}}
In order to prove the extremal aging result we consider the coarse grained process
\begin{equation}
 \tilde{S}_N(t)=\frac{1}{\tn}S_N(\nu\lfloor\frac{t\rn}{\nu}\rfloor),
\end{equation}
and prove that the convergence statement of Theorem \ref{theop1} holds for this process in $J_1$ topology.
\begin{proposition}\label{dayim} Under the assumptions of Theorem \ref{theop1}, $\mathcal{Y}$ a.s.
 \begin{equation}
\tilde{S}_N(\cdot)^{\an}\overset{N\to\infty}\longrightarrow Y_\beta(K\cdot)\text{      in  } D([0,T],J_1).
 \end{equation}
 where $K=2\beta^{-2}p$.
\end{proposition}
\begin{proof}
\newcommand{\ep}{\epsilon}
\newcommand{\cnep}{C_N(\epsilon)}

We first show that the traps from different blocks that are deeper than $\delta^{1/\an}\tn$ has a Poisson structure. With the notation as before recall that $C_N(\delta)=\an\beta^{-1}\sqrt{N}+\frac{\log(\delta)}{\an\beta\sqrt{N}}$ and define the measure $H_N^{\delta}(\text{d}x)$ on $[0,T]$ as follows
\begin{equation}
 H_N^{\delta}(\text{d}x)=\sum_{k=0}^{T\rn/\nu}\mathbf{1}\{\max_{i=k\nu+1,\dots,(k+1)\nu}X_N^0(i)\geq C_N(\delta)\}\delta_{k\nu/\rn}(\text{d}x)
\end{equation}

\begin{lemma}\label{kisrak}
 $\forall \epsilon >0$, $\mathcal{Y}$ a.s. $H_N^{\delta}$ converges to a homogeneous Poisson point process with intensity $\rho_{\delta}\in (0,\infty)$.
\end{lemma}
\begin{proof}
To prove this we use Proposition 16.17 of \cite{kal} which states that it is enough to prove that for any interval $I\subset [0,T]$, $\mathcal{Y}$ a.s.
\begin{equation}
 \lim_{N\to\infty}\mathbb{P}[H_N^{\delta}(I)=0|\mathcal{Y}]=e^{-\rho_{\delta}|I|}\;\text{ and }\;\limsup_{N\to\infty} \mathbb{E}[H_N^{\delta}(I)|\mathcal{Y}]\leq \rho_{\delta} |I|,
\end{equation}
where $|I|$ is the Lebesgue measure of $I$.

We do not need any additional estimates to prove the equations above. Take $I=[a,b]$. Note that
\begin{equation}\label{yeter}
\mathbb{P}[H_N^{\delta}(I)=0|\mathcal{Y}]=\pro[\max_{i=a\rn,\dots,b\rn}X_N^0(i)\leq C_N(\delta)].
\end{equation}
By Proposition \ref{cukkafa} we have $\mathcal{Y}$ a.s.
\begin{equation}
 \mathbb{P}[H_N^{\delta}(I)=0|\mathcal{Y}]\overset{N\to\infty} \longrightarrow e^{-\frac{K (b-a)}{\delta^{1/\beta^2}}}.
\end{equation}
Note that then it must be the case that $\rho_{\delta}=K/\delta^{1/\beta^2}$.
Considering the second condition, it is easy to see that $\mathbb{E}[H_N^{\delta}(I)|\mathcal{Y}]$ is equal to
\begin{align}
 \label{buttur}\sum_{k=\lfloor a\rn/\nu\rfloor}^{\lfloor b\rn/\nu\rfloor}\pro(\max_{k\nu+1,\dots,(k+1)\nu}X_N^0(i)\geq C_N(\delta)|\mathcal{Y}).
\end{align}

We use the block independent Gaussian process $X_N^1$. Note that for pairs $(i,j)$ in the same block we have $\Lambda_{ij}^0\geq \Lambda_{ij}^1 >0$ where $\Lambda_{ij}^1$ stands for the covariance of $X_N^1$. Thus, by the Gaussian comparison theorem, we have
\begin{align*}
 (\ref{buttur})&\leq \sum_{k=\lfloor a\rn/\nu\rfloor}^{\lfloor b\rn/\nu\rfloor}\pro(\max_{k\nu+1,\dots,(k+1)\nu}X_N^1(i)\geq C_N(\delta))\leq \frac{(b-a)\rn}{\nu}\pro(\max_{1,\dots,\nu}X_N^1(i)\geq C_N(\delta)).
\end{align*}
And by Proposition \ref{esasli} the last term above converges to $\rho_{\delta}|I|$.
\end{proof}

\newcommand{\got}{\tilde{S}_N^{\an}}
Now we can finish the proof of Proposition \ref{dayim}. Checking the convergence of finite dimensional distributions and condition (i) and the second half of (ii) of Theorem \ref{gottokiki} is completely analogous as for the original clock process $\bar{S}_N^{\an}$. Hence, we only have to prove that for any $\ep$ and $\eta$ given we can choose $\delta$ small enough so that $\mathcal{Y}$ a.s.
\begin{equation}
 \pro[w_{\got}(\delta)\geq \eta|\mathcal{Y}]\leq \epsilon
\end{equation}
for $N$ large enough.
Let
\begin{equation*}
 w_f([\tau,\tau+\delta])=\sup\{\min(|f(t_2)-f(t)|,|f(t)-f(t_1)|):\tau\leq t_1\leq t\leq t_2\leq \tau +\delta \}.
\end{equation*}
Let us define $\tilde{S}_{N,\delta}(\cdot)$ as
\begin{equation}
 \tilde{S}_{N,\delta}(t):=\tn^{-1}\sum_{i=1}^{\lfloor t\rn/\nu\rfloor \nu}e_i e^{\beta\sqrt{N} X_N^0(i)} \mathbf{1}\{e^{\beta\sqrt{N} X_N^0(i) }\leq\tn \delta^{1/\an}\}.
\end{equation}
It is clear that
\begin{equation}
 \tilde{S}_{N,\delta}(t)\leq \tn^{-1}\sum_{i=1}^{t\rn}e_i e^{\beta\sqrt{N} X_N^0(i)}\mathbf{1}\{e^{\beta\sqrt{N}X_N^0(i)}\leq\tn \delta^{1/\an}\}.
\end{equation}
Using this, for any $\epsilon$ and $\delta$ given by the equation (\ref{stan}) in the proof of Lemma \ref{godone} (setting $C=Ct/\eta/2$) $\mathcal{Y}$ a.s.
\begin{equation}
 \pro[\tilde{S}_{N,\epsilon}^{\an}(t)\geq \eta/2|\mathcal{Y} ]\leq \epsilon/2,
\end{equation}
for all $N$ large enough.

We will use the following inequality for $0\leq\alpha\leq1$ and $x,y\geq 0$
\begin{equation}\label{hadigari}
 (x+y)^{\alpha}\leq x^\alpha+y^\alpha.
\end{equation}

Consider the case $H_n^{\epsilon}([\tau,\tau+\delta])=0$. Then using (\ref{hadigari}) and the monotonicity, we have
\begin{equation}
  w_{\got}([\tau,\tau+\delta])\leq (\tilde{S}_N(\tau)+\tilde{S}_{N,\epsilon}(\tau+\delta))^{\an}-(\tilde{S}_N(\tau))^{\an}\leq \tilde{S}_{N,\epsilon}^{\an}(T)
\end{equation}
In a similar way, for the case  $H_N^{\epsilon}([\tau,\tau+\delta])= 1$ we have
\begin{equation}
 w_{\got}([\tau,\tau+\delta])\leq 2 \tilde{S}_{N,\epsilon}^{\an}(T).
\end{equation}
Using the last two inequalities above and Lemma \ref{kisrak} we can conclude that $\mathcal{Y}$ a.s.
\begin{equation*}
 \mathbb{P}[w_{\got}([\tau,\tau+\delta])\geq \eta|\tilde{S}_{N,\epsilon}^{\an}(T)\leq \eta/4,\mathcal{Y}]\leq\pro[H_N^{\epsilon}([\tau,\tau+\delta])\geq 2|\mathcal{Y}]\leq C \frac{K}{\epsilon}\delta^2.
\end{equation*}
Using
\begin{equation}
 w_{\got}(\delta)\leq \max\{w_{\got}([\tau,\tau+2\delta]):0\leq \tau \leq T,\tau=k\delta,k\in\mathbb{N}\},
\end{equation}
we get that $\mathbb{P}[w_{\got}(\delta)\geq \eta|\mathcal{Y}]$ bounded above by
\begin{align*}
 \sum_{k=0}^{T/\delta}\pro[w_{\got}([k\delta,(k+2)\delta])\geq 2|\mathcal{Y}]&\leq \pro[\got(T)\geq \eta/4|\mathcal{Y}]+\sum_{k=0}^{T/\delta}\pro[H_N^{\epsilon}([k\delta,(k+2)\delta])\geq 2|\mathcal{Y}]\\&\leq \epsilon/2+C (T/\delta) K\delta^2 /\epsilon
\end{align*}
which is less than $\epsilon$ for $\delta$ small enough. Hence, we have checked the first part of condition (ii) of Theorem \ref{gottokiki}. This finishes the proof Lemma \ref{kisrak}.

\end{proof}

\begin{proof}[Proof of Theorem \ref{theop22}]

We will actually prove the result of Theorem \ref{theop22} for a.s. $\mathcal{Y}$, that is, we will prove Theorem \ref{theo3}. Then taking the expectation over $\mathcal{Y}$ gives the result.
 
Recall that, for a fixed realization $\mathcal{Y}$, we are interested in the probability of the event
\begin{align*}
&A_N^{\epsilon}(t(N),t(N)(1+\theta)^{1/\an})=\{\dist(\sigma_N(t(N)),\sigma_N(t(N)(1+\theta)^{1/\an}))\leq N\epsilon/2\}.
\end{align*}

Let ${R}_N$ be the range of $\tilde{S}_N$. We have
\begin{equation}
\{R_N\cap (1,(1+\theta)^{1/\an})=\emptyset \}\subset A_N^{\epsilon}(\tn,\tn(1+\theta)^{1/\an})
\end{equation}
since if $\{R_N\cap (1,(1+\theta)^{1/\an})=\emptyset \}$ the random walk $\sigma_N$ makes less than $\nu$ steps in $[\tn,(1+\theta)^{1/\an}\tn]$. As a result, the overlap between $\sigma_N(\tn)$ and  $\sigma_N((1+\theta)^{1/\an}\tn)$ is O($\nu/N$).

Conversely, if $\{R_N\cap (1,(1+\theta)^{1/\an})\not=\emptyset \}$ then there exists a $u$ s.t. $\tilde{S}_N(u)\in (1,(1+\theta)^{1/\an})$, that is, $\tilde{S}_N^{\an}(u)\in(1,1+\theta)$. By Proposition \ref{dayim}, for $\eta>0$ small enough
\begin{equation*}
 \pro[\tilde{S}_N^{\an}(u+\eta)\in (1,1+\theta)]\geq 1-\delta
\end{equation*}
However, it implies that the random walk makes at least $\eta\rn$ steps and with a very high probability the overlap is 0. As result we have
\begin{equation}
 \mathbb{P}[R_N\cap (1,(1+\theta)^{1/\an})=\emptyset |\mathcal{Y}](1+o(1))= \pro [A_N^{\epsilon}(\tn,\tn(1+\theta)^{1/\an})]
\end{equation}
Since $R_N^{\an}$ is the range of $\tilde{S}_N^{\an}$ and the fact that extremal process $Y_\beta(K\cdot)$ does not hit points we have
\begin{equation}
 \pro[R_N\cap (1,(1+\theta)^{1/\an})=\emptyset |\mathcal{Y}]\overset{N\to\infty}\longrightarrow \pro[\{Y_\beta(Ku):\;u\in[0,\infty)\}\cap [t_1,t_2]=\emptyset].
\end{equation}
By the Proposition 4.8 on page 183 in \cite{res}, the range of $Y_\beta$, $\{Y_\beta(Ku):\;u\in[0,\infty)\}$ are the points of a Poisson point process on $[0,\infty)$ with mean measure $\mu((a,b))=\log(b^{1/\beta^2}/a^{1/\beta^2})$. Hence,
\begin{equation*}
\pro[\{Y_\beta(Ku):\;u\in[0,\infty)\cap (1,1+\theta)\}=\emptyset]=\exp(-\mu((1,1+\theta)))=\left(\frac{1}{1+\theta}\right)^{1/\beta^2}.
\end{equation*}
\end{proof}

\end{document}